\documentclass[11pt, a4paper]{amsart}
\usepackage{fullpage}
\usepackage{amssymb}
\usepackage{commath}
\usepackage{mathrsfs}
\usepackage{empheq}
\usepackage{amsmath, amsthm}
\usepackage{multicol}
\usepackage{parskip}
\usepackage{graphicx}
\usepackage{realboxes}

\usepackage[colorlinks=true, allcolors=blue]{hyperref}

\usepackage{tikz,tikz-3dplot,tikz-cd,tkz-tab,tkz-euclide,pgf,pgfplots}
\pgfplotsset{compat=1.18}
\usetikzlibrary{bending,quotes,angles,3d}
\usepackage{bbm}
\usepackage{url}

\newcommand{\ol}{\overline}

\renewcommand{\phi}{\varphi}
\newcommand{\llb}{\left\lbrace}
\newcommand{\rrb}{\right\rbrace}

\newcommand{\llv}{\left\lvert}
\newcommand{\rrv}{\right\rvert}
\newcommand{\1}{\mathbbm{1}}

\newcommand{\mrev}[1]{\href{http://mathscinet.ams.org/mathscinet/article?mr=#1}{MR#1}}
\newcommand{\zbl}[1]{\href{http://zbmath.org/#1}{Zbl #1}}

\newcommand{\doi}[1]{doi:\,\href{http://dx.doi.org/#1}{#1}}

\DeclareMathOperator{\Tor}{Tor}
\DeclareMathOperator{\Ext}{Ext}
\DeclareMathOperator{\Hom}{Hom}

\newtheorem{thm}{Theorem}[section]

\theoremstyle{plain}
\newtheorem{prop}[thm]{Proposition}
\newtheorem{lem}[thm]{Lemma}
\newtheorem{cor}[thm]{Corollary}

\newtheorem*{theorem*}{Theorem}

\theoremstyle{definition}
\newtheorem{defn}[thm]{Definition}

\theoremstyle{remark}
\newtheorem{rem}[thm]{Remark}
\newtheorem{eg}[thm]{Example}

\makeatletter
\@namedef{subjclassname@1991}{\textup{2020} Mathematics Subject Classification}
\makeatother

\title{Cohomology of rook-Brauer algebras and their subalgebras}
\author{Andrew Fisher}
\address{(Andrew Fisher) School of Mechanical, Aerospace and Civil Engineering, Sir Frederick Mappin Building, University of Sheffield, Mappin Street, S1 4DT, UK}
\email{andrew.fisher@sheffield.ac.uk}
\author{Daniel Graves}
\address{(Daniel Graves) Lifelong Learning Centre, University of Leeds, Woodhouse, Leeds, LS2 9JT, UK}
\email{dan.graves92@gmail.com}

\begin{document}



\begin{abstract}
This paper studies the (co)homology of rook-Brauer algebras and their subalgebras. Our main results focus on the cohomology of rook-Brauer algebras (which is related to the cohomology of symmetric groups), the cohomology of Motzkin algebras (for which we obtain a vanishing result in positive degrees) and the cohomology of walled Brauer algebras (which is related to the cohomology of products of symmetric groups). Along the way we collect some cohomological analogues of known results for Temperley--Lieb algebras, Brauer algebras and rook algebras.
\end{abstract}

\maketitle

\section{Introduction}

\emph{Brauer algebras} were first introduced in \cite{Brauer}, where they were shown to exhibit a Schur--Weyl duality with the orthogonal groups. Since then, one of the most common presentations of the Brauer algebras is in the form of \emph{diagram algebras}. A \emph{Brauer $n$-diagram} is a graph on two columns of $n$ vertices such that each vertex is connected to precisely one other by an edge. One forms an algebra, $\mathcal{B}_n(\delta)$, from such diagrams by taking $k$-linear combinations (for some commutative ground ring, $k$). The product is induced from gluing two diagrams together along a column of vertices and replacing any loops in the middle column by a parameter $\delta$ in the ground ring (precise definitions given below in Section \ref{alg-sec}). 

A \emph{rook-Brauer $n$-diagram} (sometimes called a \emph{partial Brauer $n$-diagram}) is a graph on two columns of $n$ vertices such that each vertex is connected to \emph{at most} one other by an edge. As above, we can form an algebra, $\mathcal{RB}_n(\delta,\varepsilon),$ with basis the rook-Brauer $n$-diagrams. The key difference, since we allow isolated vertices in our graphs, is that we can obtain two topologically different types of component within the middle column. We can obtain contractible components (isolated vertices and edges that lie only within the middle column) and we can form loops as before (see Example \ref{comp-eg} below for an example of this composition). The parameter $\delta$, as for the Brauer algebras, replaces any loops formed in the composition of diagrams. The parameter $\varepsilon$ replaces any contractible components in the middle column of a composite of two diagrams.

Rook-Brauer algebras are also of interest from the point of view of representation theory. They exhibit a Schur--Weyl duality with the orthogonal groups \cite{HD,MM} and, more generally, with twin groups \cite{DG-twin}.

The rook-Brauer algebras also contain a number of interesting subalgebras spanned by subsets of their diagram basis. We briefly recall the subalgebras that we will study in this paper, together with a short history.

The \emph{Brauer algebra}, $\mathcal{B}_n(\delta)$, described above, is a subalgebra and there is a vast literature on these. Note that since we do not allow isolated vertices, we can drop $\varepsilon$ from the notation of this subalgebra. The same is true for the Temperley--Lieb algebras and walled Brauer algebras described below.
    
The \emph{Temperley--Lieb algebra}, $\mathcal{TL}_n(\delta)$, is spanned by the Brauer diagrams that are also planar in the sense that all edges are contained within the rectangle formed by the vertices and no edges intersect. They are of interest in statistical mechanics, knot theory and representation theory \cite{TL,Jones1,Jones2,Westbury}.

The \emph{rook algebra}, $\mathcal{R}_n(\varepsilon)$, is spanned by diagrams such that a vertex is either isolated or connected to a vertex in the other column. The \emph{planar rook algebra}, $\mathcal{PR}_n(\varepsilon)$ is spanned by those rook diagrams which are planar. The representation theory of these algebras have been studied in \cite{HalvR, Xiao, Campbell,FHH-planar-rook,BM}. Since we cannot have edges between two vertices in the same column, we cannot form any loops. Hence we drop $\delta$ from the notation in these algebras.

The \emph{Motzkin algebras}, $\mathcal{M}_n(\delta,\varepsilon)$, are spanned by rook-Brauer diagrams that are also planar. They were first introduced in the paper \cite{BHal}. They exhibit a Schur--Weyl duality with the special linear Lie algebra of order $2$ over $k$ (see \cite{BHal} and also \cite[Sections 1, 9, 10]{Doty_Giaquinto}). Motzkin algebras have also been used to construct a tensor category with fusion rules of type $A$ \cite{JoYa}.

The \emph{walled Brauer algebras}, $\mathcal{B}_{r,s}(\delta)$, were introduced as a centralizer of a diagonal action of a general linear group on a tensor space (see \cite{Koike,Turaev,BCHLLS,Nikitin} for instance). Their representation theory has been well studied (see \cite{Halverson-walled,CDDM,JK,CVPSW} for instance). They are related to Khovanov's arc algebra \cite{Brundan-Stroppel} and they also arise in quantum physics \cite{BJSH}. Walled Brauer $(r+s)$-diagrams share the property of Brauer diagrams that every vertex is connected to precisely one other but we now have extra conditions. We have a ``semi-porous wall" separating the first $r$ vertices in the two columns from the last $s$ vertices in the two columns. Any edge between the two columns is forbidden from crossing the wall but any edge between two vertices in the same column is required to cross the wall.

The aim of this paper is to study the (co)homology of the rook-Brauer algebras and those subalgebras which have not yet appeared in the literature. Each of the algebras mentioned above can be equipped with an augmentation so we can define their homology and cohomology as certain $\Tor$ and $\Ext$ groups following Benson \cite[Definition 2.4.4]{Benson1}.

The homology of the Temperley--Lieb algebras and the Brauer algebras has been well studied (see \cite{BH1,Sroka,Boyde,BBRWS,Boyde-planar} for the Temperley--Lieb case and \cite{BHP,Boyde} for the Brauer algebra case). These form part of a larger literature on the (co)homology of diagram algebras which also includes the \emph{partition algebras} and variations such as the \emph{Tanabe algebras} and the \emph{coloured partition algebras} \cite{BHP2,Boyde2,Fisher-Graves, Cranch-Graves}, the \emph{Jones annular algebras} \cite{Boyde2} and the \emph{dilute Temperley--Lieb algebras} \cite{FG-dilute}.

For the Brauer algebra, $\mathcal{B}_n(\delta)$, it is known that for any $\delta$, the homology is isomorphic to the homology of the symmetric group $\Sigma_n$ if $n$ is odd \cite[Theorem 1.3]{Boyde}. If $n$ is even, such an isomorphism is only true in a range \cite[Theorem B]{BHP}. It is known that the range is sharp for $n=2$ but this is not known for larger even numbers.

Our first main result concerns the walled Brauer algebras. Our result takes a different form to the Brauer algebra, not depending directly on the parity of the indices, but depending on whether the indices $r$ and $s$ are equal. We prove that for the walled Brauer algebra, $\mathcal{B}_{r,s}(\delta)$, the (co)homology can be identified with the (co)homology of the product of symmetric groups $\Sigma_r\times \Sigma_s$ for any $\delta$ when $r\neq s$ and that the isomorphism holds in a range when $r=s$. Our main result is as follows (Theorem \ref{WB-cohom-thm} in the text):

\begin{theorem*}
Let $r,s\geqslant 1$. For any $\delta\in k$ and for $r\neq s$, there exist isomorphisms of graded $k$-modules
\[\Tor_{\star}^{\mathcal{B}_{r,s}(\delta)}(\1,\1) \cong H_{\star}(\Sigma_{r}\times \Sigma_s,\1) \quad \text{and} \quad \Ext_{\mathcal{B}_{r,s}(\delta)}^{\star}(\1,\1) \cong H^{\star}(\Sigma_{r}\times \Sigma_s,\1).\]
Furthermore, when $r=s$, these isomorphisms hold in the range $0\leqslant \star\leqslant \frac{r+s}{2}-1$.   
\end{theorem*}

We show that the range given when $r=s$ cannot be improved in general, by showing that it is sharp when $r=s=1$. In this case, the walled Brauer algebra $\mathcal{B}_{1,1}(\delta)$ is isomorphic to the Temperley--Lieb algebra $\mathcal{TL}_2(\delta)$ so we can deduce sharpness from the calculations of Boyd and Hepworth \cite[Proposition 7.1]{BH1}. We also show that the isomorphism of graded $k$-modules holds for any $r$ and $s$ when the parameter $\delta$ is invertible in $k$ (Theorem \ref{walled-Brauer-thm} in the text). The proof uses idempotent covers and the Mayer--Vietoris complex from \cite{Boyde2}.

Our second main result concerns the rook-Brauer algebras and the Motzkin algebras. There already exist statements on the homology of rook-Brauer algebras in the literature, namely Theorem 7.3 and Corollary 7.4 of \cite{Boyde}. These results claim that if $\varepsilon$ is invertible in $k$ and $\delta \in k$ is any element, then the homology of the rook-Brauer algebra $\mathcal{RB}_n(\delta,\varepsilon)$ is isomorphic to the homology of the Brauer algebra $\mathcal{B}_n(\delta)$. The justification for this claim occurs before the statement of Theorem 7.3 and goes as follows. In Section 5 \emph{in loc.~cit.}, Boyde defines rook-Brauer $n$-diagrams $\rho_i$, which are formed from the rook-Brauer $n$-diagram with $n$ horizontal edges by omitting the $i^{\mathrm{th}}$ edge. For example, the diagrams $\rho_1$ and $\rho_2$ in $\mathcal{RB}_2(\delta,\varepsilon)$ are pictured below:

\begin{center}
\begin{multicols}{2}
\begin{tikzpicture}
\fill (0,0) circle[radius=2pt];
\fill (0,1) circle[radius=2pt];
\fill (1,0) circle[radius=2pt];
\fill (1,1) circle[radius=2pt];  
\path[-] (0,0) edge (1,0);
\end{tikzpicture}
\\
\begin{tikzpicture}
\fill (0,0) circle[radius=2pt];
\fill (0,1) circle[radius=2pt];
\fill (1,0) circle[radius=2pt];
\fill (1,1) circle[radius=2pt];  
\path[-] (0,1) edge (1,1);
\end{tikzpicture}
\end{multicols}
\end{center}

Let $I$ be the two-sided ideal in $\mathcal{RB}_n(\delta,\varepsilon)$ generated by the elements $\rho_i$. It is implicitly claimed that the quotient $\mathcal{RB}_n(\delta , \varepsilon)/I$ is isomorphic as an algebra to the Brauer algebra $\mathcal{B}_n(\delta)$. However, this is not the case. We have the following product in $\mathcal{RB}_2(\delta,\varepsilon)$, which lies in the two-sided ideal $I$ since we are multiplying on the left and right of $\rho_2$:

\begin{center}
\begin{tikzpicture}
\fill (0,0) circle[radius=2pt];
\fill (0,1) circle[radius=2pt];
\fill (1,0) circle[radius=2pt];
\fill (1,1) circle[radius=2pt];
\path[-] (0,0) edge [bend right=20] (0,1);

\draw (1.5,0.5) node {$\cdot$};

\fill (2,0) circle[radius=2pt];
\fill (2,1) circle[radius=2pt];
\fill (3,0) circle[radius=2pt];
\fill (3,1) circle[radius=2pt];
\path[-] (2,1) edge (3,1);

\draw (3.5,0.5) node {$\cdot$};

\fill (4,0) circle[radius=2pt];
\fill (4,1) circle[radius=2pt];
\fill (5,0) circle[radius=2pt];
\fill (5,1) circle[radius=2pt];
\path[-] (5,0) edge [bend left=20] (5,1);

\draw (5.5,0.5) node {$=$};
\draw (6,0.5) node {$\varepsilon^3$};

\fill (6.5,0) circle[radius=2pt];
\fill (6.5,1) circle[radius=2pt];
\fill (7.5,0) circle[radius=2pt];
\fill (7.5,1) circle[radius=2pt];
\path[-] (6.5,0) edge [bend right=20] (6.5,1);
\path[-] (7.5,0) edge [bend left=20] (7.5,1);
\end{tikzpicture}
\end{center}

The diagram on the right is a scalar multiple of a Brauer $2$-diagram. Since $\varepsilon$ is invertible, we can scale this to show that the Brauer $2$-diagram 
\begin{center}
\begin{tikzpicture}
\fill (0,0) circle[radius=2pt];
\fill (0,1) circle[radius=2pt];
\fill (1,0) circle[radius=2pt];
\fill (1,1) circle[radius=2pt];  
\path[-] (0,0) edge [bend right=20] (0,1);
\path[-] (1,0) edge [bend left=20] (1,1);
\end{tikzpicture}
\end{center}
also lies in $I$ so the isomorphism of algebras does not hold.

We show that if $\varepsilon$ is invertible then, for any $\delta$, the (co)homology of the rook-Brauer algebra $\mathcal{RB}_n(\delta, \varepsilon)$ is isomorphic to the (co)homology of the symmetric group $\Sigma_n$. Our methods can also be used to show that, under the same conditions, the (co)homology of Motzkin algebras vanishes in positive degrees. The following is a combination of Theorems \ref{RB-thm} and \ref{Motzkin-thm}.

\begin{theorem*}
Let $\varepsilon$ be invertible. Then, for any $\delta$, there exist isomorphisms of graded $k$-modules
\[\Tor_{\star}^{\mathcal{RB}_n(\delta,\varepsilon)}(\1 ,\1) \cong H_{\star}(\Sigma_n ,\1) \quad \text{and} \quad \Ext_{\mathcal{RB}_n(\delta,\varepsilon)}^{\star}(\1 ,\1) \cong H^{\star}(\Sigma_n ,\1)\]
and the graded $k$-modules $\Tor_{\star}^{\mathcal{M}_n(\delta,\varepsilon)}(\1 ,\1)$ and $\Ext_{\mathcal{M}_n(\delta,\varepsilon)}^{\star}(\1 ,\1)$ are isomorphic to $k$ concentrated in degree zero.
\end{theorem*}

\begin{rem}
    Since the submission of this paper, Khoa Ta \cite{KhoaTa} has found an alternative proof of the homological part of this result using \emph{inductive resolutions}, after the fashion of \cite{BH1,BHP,BHP2}.
\end{rem}

Along the way we also collect a couple of cohomological analogues of known results for the homology of Brauer algebras, Temperley--Lieb algebras and rook algebras found in \cite{Boyde}. We also prove the related results for the planar rook algebras. These results follow methods introduced in \cite{Boyde}. 

In Section \ref{Boyde-sec}, we prove the cohomological analogue of \cite[Theorem 1.11]{Boyde}: if a subalgebra of the rook-Brauer algebra is free, as a module, on a basis of diagrams and if certain left ideals (based on \emph{link states}, see Section \ref{link-state-sec}) are principal and generated by an idempotent then we can deduce a chain of isomorphisms on cohomology, depending on the number of propagating edges.

The paper is structured as follows. In Section \ref{alg-sec}, we recall the definition of the rook-Brauer algebras, together with the subalgebras introduced above. In Section \ref{augmentation-sec} we define augmentations for each of these algebras and define their (co)homology in terms of the augmentation. In Section \ref{link-state-sec} we recall the notion of a link state for a rook-Brauer diagram and use this to define families of ideals which we will use to prove our main results. We also recall the notions of double diagram and sesqui-diagram. 

In Section \ref{Boyde-sec}, we prove the cohomological analogues of \cite[Theorem 4.3]{Boyde} and \cite[Corollary 4.4]{Boyde} and provide an application in the case of subalgebras of the rook-Brauer algebras. This will be our main technical result for studying the cohomology of the rook-Brauer algebras and the Motzkin algebras. In Section \ref{known-results-sec}, we gather cohomological analogues of results for the rook algebras with $\varepsilon$ invertible \cite[Theorem 5.4]{Boyde} and for Brauer algebras and Temperley--Lieb algebras with $n$ odd (see \cite[Theorems 1.1, 1.3]{Boyde} and \cite[Theorem A]{Sroka}). We also include results on the (co)homology of planar rook algebras. In Section \ref{RB-M-sec}, we prove that if $\varepsilon$ is invertible then the (co)homology of the rook-Brauer algebra $\mathcal{RB}_n(\delta,\varepsilon)$ is isomorphic to the (co)homology of the symmetric group $\Sigma_n$ and the (co)homology of the Motzkin algebras vanishes in positive degrees.

In Section \ref{WB-sec} we show that the (co)homology of the walled Brauer algebra $\mathcal{B}_{r,s}(\delta)$ is isomorphic to the cohomology of $\Sigma_r\times \Sigma_s$ when the parameter $\delta$ is invertible. We then construct a $k$-free idempotent left cover of the two-sided ideal described above and use this to show that the (co)homology of the walled Brauer algebra $\mathcal{B}_{r,s}(\delta)$ is isomorphic to the cohomology of $\Sigma_r\times \Sigma_s$ when $r\neq s$ for any $\delta \in k$ and that such an isomorphism holds in a range when $r=s$ is even.

\subsection*{Acknowledgements}
We would like to thank both James Brotherston and Natasha Cowley for helpful and interesting conversations whilst writing this paper. We would like to thank James Cranch and Sarah Whitehouse for their feedback and support in this project and related work. We'd like to thank Rachael Boyd and Richard Hepworth-Young for interesting conversations at the 2024 British Topology Meeting in Aberdeen. We would like to thank Guy Boyde for his comments and feedback on previous drafts of this paper. We would like to thank the anonymous referee for their helpful comments.

\subsection{Conventions}
Throughout, $k$ will be a unital, commutative ring and $n$ will be a positive integer unless otherwise stated.

\section{Rook-Brauer algebras and their subalgebras}
\label{alg-sec}

In this section, we recall the definition of the rook-Brauer algebras, the Motzkin algebras, the (planar) rook algebras, the (walled) Brauer algebras and the Temperley--Lieb algebras.

\begin{defn}
\label{rb-diag-defn}
A \emph{rook-Brauer $n$-diagram} is a graph on two columns of $n$ vertices where each edge is incident to two distinct vertices, there is at most one edge between any two vertices and any vertex is connected to at most one other vertex.

By convention, the vertices down the left-hand column will be labelled by $1,\dots , n$ in ascending order from top to bottom and the vertices down the right-hand column will be labelled by $\ol{1},\dotsc , \ol{n}$ in ascending order from top to bottom. Furthermore, we assume that all edges are contained within the convex hull of the vertices.
\end{defn}

\begin{defn}
\label{terminology-defn}
We collect some important terminology for graphs that will be used throughout the rest of the paper.
\begin{enumerate}
    \item An edge that connects the left-hand column of vertices to the right-hand column of vertices will be called a \emph{propagating edge}. 
    \item An edge that connects two vertices in the same column will be called a \emph{non-propagating edge}. 
    \item A vertex not connected to any other by an edge will be called an \emph{isolated vertex}.
    \item An $n$-diagram having precisely $n$ propagating edges will be called a \emph{permutation diagram}. All other $n$-diagrams will be referred to as \emph{non-permutation diagrams}.
\end{enumerate}
\end{defn}

\begin{defn}
Let $\delta,\, \varepsilon \in k$. The \emph{rook-Brauer algebra}, $\mathcal{RB}_n(\delta,\varepsilon)$, is the $k$-algebra with basis consisting of all rook-Brauer $n$-diagrams with the multiplication defined by the $k$-linear extension of the following product of diagrams. Let $d_1$ and $d_2$ be rook-Brauer $n$-diagrams. The product $d_1d_2$ is obtained by the following procedure:
\begin{itemize}
    \item Place the diagram $d_2$ to the right of the diagram $d_1$ and identify the vertices $\ol{1},\dotsc , \ol{n}$ in $d_1$ with the vertices $1,\dotsc , n$ in $d_2$. Call this diagram $d_1\ast d_2$. We drop the labels of the vertices in the middle column and we preserve the labels of the left-hand column and right-hand column.
    \item Count the number of loops that lie entirely within the middle column. Call this number $\alpha$. Count the number of contractible connected components that lie entirely within the middle column (that is, isolated vertices in the middle column or an edge which is connected to neither the left-hand column nor the right-hand column). Call this number $\beta$.
    \item Make a new rook-Brauer $n$-diagram, $d_3$, as follows. Given distinct vertices $x$ and $y$ in the set $\llb 1, \ol{1}, \dotsc , n , \ol{n}\rrb$, $d_3$ has an edge between $x$ and $y$ if there is a path from $x$ to $y$ in $d_1\ast d_2$.
    \item We define $d_1d_2=\delta^{\alpha}\varepsilon^{\beta}d_3$.
\end{itemize}
The product is associative and the identity element consists of the diagram with $n$ horizontal edges.
\end{defn}

\begin{eg}
\label{comp-eg}
Here is an example of composition in $\mathcal{RB}_4(\delta,\varepsilon)$. We drop the labels on the vertices for a clearer picture. In the diagram below, we form one loop in the middle column and we have one isolated vertex in the middle column and so we obtain a factor of $\delta\varepsilon$.
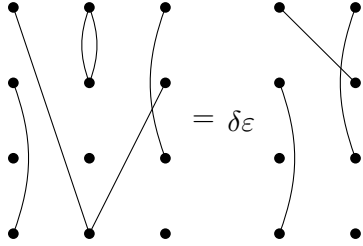
\begin{center}
    \begin{tikzpicture}
\fill (0,0) circle[radius=2pt];
\fill (0,1) circle[radius=2pt];
\fill (0,2) circle[radius=2pt];
\fill (0,3) circle[radius=2pt];
\fill (1,0) circle[radius=2pt];
\fill (1,1) circle[radius=2pt];
\fill (1,2) circle[radius=2pt];
\fill (1,3) circle[radius=2pt];
\fill (2,0) circle[radius=2pt];
\fill (2,1) circle[radius=2pt];
\fill (2,2) circle[radius=2pt];
\fill (2,3) circle[radius=2pt];
\fill (3.5,0) circle[radius=2pt];
\fill (3.5,1) circle[radius=2pt];
\fill (3.5,2) circle[radius=2pt];
\fill (3.5,3) circle[radius=2pt];
\fill (4.5,0) circle[radius=2pt];
\fill (4.5,1) circle[radius=2pt];
\fill (4.5,2) circle[radius=2pt];
\fill (4.5,3) circle[radius=2pt];
\draw (2.5,1.5) node {$=$};
\draw (3,1.5) node {$\delta \varepsilon$};
\path[-] (0,0) edge [bend right=20] (0,2);
\path[-] (1,2) edge [bend right=20] (1,3);
\path[-] (1,2) edge [bend left=20] (1,3);
\draw (0,3) -- (1,0) -- (2,2);
\path[-] (2,1) edge [bend left=20] (2,3);
\path[-] (3.5,0) edge [bend right=20] (3.5,2);
\draw (3.5,3) -- (4.5,2);
\path[-] (4.5,1) edge [bend left=20] (4.5,3);
    \end{tikzpicture}
\end{center}

We refer the reader to the papers \cite{HD} and \cite{Boyde} for further examples of composing rook-Brauer diagrams.   
\end{eg}

\begin{defn}
\label{Motzkin-defn}
Fix $n\geqslant 1$.
\begin{enumerate}
    \item A \emph{Motzkin $n$-diagram} is a planar rook-Brauer $n$-diagram, that is, a rook-Brauer $n$-diagram such that no edge crosses another. The \emph{Motzkin algebra}, $\mathcal{M}_n\left(\delta, \varepsilon\right)$, is the subalgebra of $\mathcal{RB}_n(\delta,\varepsilon)$ spanned $k$-linearly by the Motzkin $n$-diagrams.
    \item A \emph{rook $n$-diagram} is a rook-Brauer $n$-diagram such that each connected component is either an isolated vertex, or it consists of exactly one vertex from the left-hand column and one vertex from the right-hand column. A \emph{planar rook $n$-diagram} is a rook $n$-diagram which is planar. The \emph{rook algebra}, $\mathcal{R}_n(\varepsilon)$, is the subalgebra of $\mathcal{RB}_n(\delta,\varepsilon)$ spanned $k$-linearly by the rook $n$-diagrams. The \emph{planar rook algebra}, $\mathcal{PR}_n(\varepsilon)$, is the subalgebra of $\mathcal{RB}_n(\delta,\varepsilon)$ spanned $k$-linearly by the planar rook $n$-diagrams. 
    \item A \emph{Brauer $n$-diagram} is a rook-Brauer $n$-diagram with no isolated vertices. The \emph{Brauer algebra}, $\mathcal{B}_n(\delta)$, is the subalgebra of $\mathcal{RB}_n(\delta,\varepsilon)$ spanned $k$-linearly by the Brauer $n$-diagrams.
    \item A \emph{Temperley--Lieb $n$-diagram} is a planar rook-Brauer $n$-diagram with no isolated vertices. The \emph{Temperley--Lieb algebra}, $\mathcal{TL}_n(\delta)$, is the subalgebra of $\mathcal{RB}_n(\delta,\varepsilon)$ spanned $k$-linearly by the Temperley--Lieb $n$-diagrams.
\end{enumerate}
\end{defn}

\begin{defn}
Let $r$ and $s$ be non-negative integers. A \emph{walled Brauer $(r+s)$-diagram} is a Brauer $(r+s)$-diagram such that
\begin{itemize}
    \item a propagating edge must either join one of the first $r$ vertices of the left-hand column with one of the first $r$ vertices of the right-hand column or join one of the last $s$ vertices of the left-hand column with one of the last $s$ vertices of the right-hand column,
    \item a non-propagating edge in either column must join one of the first $r$ vertices with one of the last $s$ vertices.
\end{itemize}
\end{defn}

\begin{eg}
\label{WB-eg}
Consider the two diagrams below. We have drawn in a dashed horizontal line to indicate the division of the vertices into two parts.
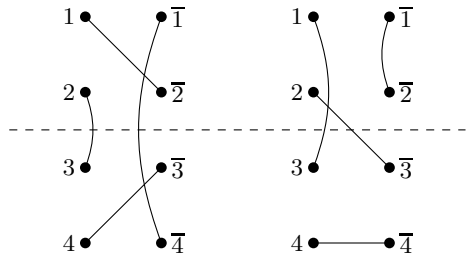
\begin{center}
    \begin{tikzpicture}
\fill (0,0) circle[radius=2pt];
\fill (0,1) circle[radius=2pt];
\fill (0,2) circle[radius=2pt];
\fill (0,3) circle[radius=2pt];
\fill (1,0) circle[radius=2pt];
\fill (1,1) circle[radius=2pt];
\fill (1,2) circle[radius=2pt];
\fill (1,3) circle[radius=2pt];
\fill (3,0) circle[radius=2pt];
\fill (3,1) circle[radius=2pt];
\fill (3,2) circle[radius=2pt];
\fill (3,3) circle[radius=2pt];   
\fill (4,0) circle[radius=2pt];
\fill (4,1) circle[radius=2pt];
\fill (4,2) circle[radius=2pt];
\fill (4,3) circle[radius=2pt]; 

\draw (0,0) node[left] {\footnotesize $4$};
\draw (0,1) node[left] {\footnotesize $3$};
\draw (0,2) node[left] {\footnotesize $2$};
\draw (0,3) node[left] {\footnotesize $1$};
\draw (1,0) node[right] {\footnotesize $\ol{4}$};
\draw (1,1) node[right] {\footnotesize $\ol{3}$};
\draw (1,2) node[right] {\footnotesize $\ol{2}$};
\draw (1,3) node[right] {\footnotesize $\ol{1}$};
\draw (3,0) node[left] {\footnotesize $4$};
\draw (3,1) node[left] {\footnotesize $3$};
\draw (3,2) node[left] {\footnotesize $2$};
\draw (3,3) node[left] {\footnotesize $1$};
\draw (4,0) node[right] {\footnotesize $\ol{4}$};
\draw (4,1) node[right] {\footnotesize $\ol{3}$};
\draw (4,2) node[right] {\footnotesize $\ol{2}$};
\draw (4,3) node[right] {\footnotesize $\ol{1}$};

\draw[dashed] (-1,1.5) -- (5,1.5);

\draw (0,0) -- (1,1);
\draw (0,3) -- (1,2);
\path[-] (0,1) edge [bend right=20] (0,2);
\path[-] (1,0) edge [bend left=20] (1,3);
\draw (3,0) -- (4,0);
\draw (3,2) -- (4,1);
\path[-] (3,1) edge [bend right=20] (3,3);
\path[-] (4,2) edge [bend left=20] (4,3);
    \end{tikzpicture}
\end{center}

The diagram on the left is a walled Brauer $(2+2)$-diagram. The diagram on the right is not: there is a propagating edge that crosses the dividing line and a non-propagating edge in the right-hand column which does not. Both of these prevent this diagram from being a walled Brauer $(2+2)$-diagram.
\end{eg}

\begin{defn}
 Let $r$ and $s$ be non-negative integers. The \emph{walled Brauer algebra} $\mathcal{B}_{r,s}(\delta)$ is the subalgebra of the Brauer algebra $\mathcal{B}_{r+s}(\delta)$ (and therefore a subalgebra of $\mathcal{RB}_{r+s}(\delta,\varepsilon)$) spanned $k$-linearly by the walled Brauer $(r+s)$-diagrams. We note that $\mathcal{B}_{r,0}\cong \mathcal{B}_{0,r} \cong k[\Sigma_r]$. 
\end{defn}

\section{Cohomology of augmented algebras}
\label{augmentation-sec}

Recall that a $k$-algebra is said to be \emph{augmented} if it comes equipped with a $k$-algebra map $\tau\colon A \rightarrow k$, which is called the \emph{augmentation}. (We note that an augmentation is usually denoted by $\varepsilon$, but we are already using this as a parameter in some of our algebras.)

We begin by defining a collection of two-sided ideals inside each of the algebras defined in Section \ref{alg-sec}. Each of the definitions made below is seen to be valid by observing that the composite of two diagrams with $i$ propagating edges and $j$ propagating edges respectively is a scalar multiple of a diagram with at most $\min(i,j)$ propagating edges (see \cite[Section 2.2]{HD}).

\begin{defn}
\label{RB-ideal-defn}
For $0\leqslant i \leqslant n-1$, let $I_i$ denote the two-sided ideal of $\mathcal{RB}_n(\delta , \varepsilon)$ spanned $k$-linearly by rook-Brauer $n$-diagrams having at most $i$ propagating edges. We will use the fact that $I_{-1}=0$.

\end{defn}

Let $A$ denote any of the subalgebras of $\mathcal{RB}_n(\delta , \varepsilon)$ defined in Section \ref{alg-sec}. We can equip $A$ with an augmentation. Since $I_{n-1}$ is a two-sided ideal in $\mathcal{RB}_n(\delta , \varepsilon)$ (spanned $k$-linearly by all non-permutation diagrams), we see that $A\cap I_{n-1}$ is a two-sided ideal in $A$ spanned $k$-linearly by all non-permutation diagrams in $A$. 

\begin{defn}
Let $A$ be any subalgebra of a rook-Brauer algebra defined in Section \ref{alg-sec}. We equip $A$ with the augmentation $\tau\colon A \rightarrow k$ that sends the permutation diagrams to $1 \in k$ and all non-permutation diagrams to $0\in k$.
\end{defn}

\begin{defn}
Let $A$ denote any of the algebras defined in Section \ref{alg-sec}. The \emph{trivial module}, $\mathbbm{1}$, consists of a single copy of the ground ring $k$, where $A$ acts on $k$ via the augmentation.
\end{defn}

We will study the homology and cohomology of the algebras introduced in Section \ref{alg-sec}. In other words, for any such algebra $A$, we are going to consider the graded $k$-modules
\[\Tor_{\star}^{A}(\1 , \1) \quad \text{and} \quad \Ext_A^{\star}(\1 ,\1).\]

\section{Link states and ideals}
\label{link-state-sec}
In this section we recall the notion of a \emph{link state}. This material follows \cite[Section 1]{Boyde}. Link states were first introduced in the paper \cite{RSA}. They in turn point out that the notion is related to \emph{parenthesis structures} \cite{Kauffmann1}, \emph{arch configurations} \cite{DGG} and \emph{cellular structure} \cite{GL-cellular}.

\subsection{Link states for rook-Brauer diagrams}
We begin by recalling link states for rook-Brauer algebras and their subalgebras.

\begin{defn}
By slicing vertically down the middle of a rook-Brauer $n$-diagram we obtain its \emph{left link state} and \emph{right link state}. Explicitly, we split all propagating edges at their midpoint and preserve all non-propagating edges. By restriction this gives us the notion of link states for Motzkin $n$-diagrams, (walled) Brauer $n$-diagrams and Temperley--Lieb $n$-diagrams. 
\end{defn}

\begin{rem}
The right link state of a rook-Brauer $n$-diagram consists of a column of $n$ vertices, labelled $\ol{1},\dotsc \ol{n}$, such that at each vertex we have one of the following three situations:
\begin{itemize}
    \item the vertex has a hanging edge, called a \emph{defect};
    \item the vertex is connected to precisely one other vertex by a non-propagating edge;
    \item the vertex is an isolated vertex.
\end{itemize}
We note that link states for Brauer $n$-diagrams, Temperley--Lieb $n$-diagrams and walled Brauer $(r+s)$-diagrams will not have isolated vertices. Link states for Temperley--Lieb $n$-diagrams and Motzkin $n$-diagrams must be planar.
\end{rem}

\begin{eg}
\label{WB-link-eg}
In the diagram below we see the walled Brauer $(2+2)$-diagram from Example \ref{WB-eg}, together with its right link state:
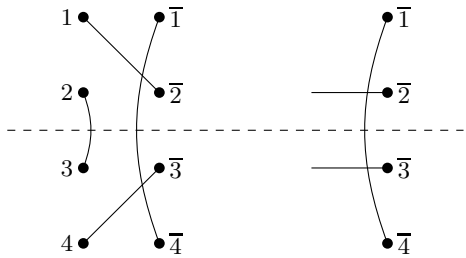
\begin{center}
    \begin{tikzpicture}
\fill (0,0) circle[radius=2pt];
\fill (0,1) circle[radius=2pt];
\fill (0,2) circle[radius=2pt];
\fill (0,3) circle[radius=2pt];
\fill (1,0) circle[radius=2pt];
\fill (1,1) circle[radius=2pt];
\fill (1,2) circle[radius=2pt];
\fill (1,3) circle[radius=2pt];   
\fill (4,0) circle[radius=2pt];
\fill (4,1) circle[radius=2pt];
\fill (4,2) circle[radius=2pt];
\fill (4,3) circle[radius=2pt]; 

\draw (0,0) node[left] {\footnotesize $4$};
\draw (0,1) node[left] {\footnotesize $3$};
\draw (0,2) node[left] {\footnotesize $2$};
\draw (0,3) node[left] {\footnotesize $1$};
\draw (1,0) node[right] {\footnotesize $\ol{4}$};
\draw (1,1) node[right] {\footnotesize $\ol{3}$};
\draw (1,2) node[right] {\footnotesize $\ol{2}$};
\draw (1,3) node[right] {\footnotesize $\ol{1}$};
\draw (4,0) node[right] {\footnotesize $\ol{4}$};
\draw (4,1) node[right] {\footnotesize $\ol{3}$};
\draw (4,2) node[right] {\footnotesize $\ol{2}$};
\draw (4,3) node[right] {\footnotesize $\ol{1}$};

\draw[dashed] (-1,1.5) -- (5,1.5);

\draw (0,0) -- (1,1);
\draw (0,3) -- (1,2);
\path[-] (0,1) edge [bend right=20] (0,2);
\path[-] (1,0) edge [bend left=20] (1,3);
\draw (4,1) -- (3,1);
\draw (4,2) -- (3,2);
\path[-] (4,0) edge [bend left=20] (4,3);
    \end{tikzpicture}
\end{center}
\end{eg}

\begin{defn}
For $0\leqslant i \leqslant n$, we let $P_i$ denote the set of right link states of rook-Brauer $n$-diagrams with precisely $i$ defects.
\end{defn}

The set $P_i$ consists of the right link states of the diagrams which span the quotient ideal $I_i/I_{i-1}$. This is an important part of the proof of \cite[Corollary 4.4]{Boyde} and Corollary \ref{Ext-cor}.

\begin{defn}
Suppose we have a right link state of a rook-Brauer $n$-diagram. We can remove two defects and replace them with a non-propagating edge joining the two vertices. This operation is called a \emph{splice}. We can also remove a defect, leaving an isolated vertex. This operation is called a \emph{deletion}.
\end{defn}

\begin{rem}
These operations continue to be valid for subsets of rook-Brauer $n$-diagrams, but sometimes with conditions. We do not allow the deletion of defects when considering the (walled) Brauer algebra or the Temperley--Lieb algebra as this would result in isolated vertices. When considering the Temperley--Lieb algebra and the Motzkin algebra, we do not allow splices that would violate the planarity condition. In the case of the walled Brauer algebra, we only allow splices that connect a defect at one of the first $r$ vertices with a defect at one of the last $s$ vertices.  
\end{rem}

\begin{eg}
The diagrams below show the right link state of the walled Brauer $(2+2)$-diagram from Example \ref{WB-link-eg}, together with the result of splicing the two defects.
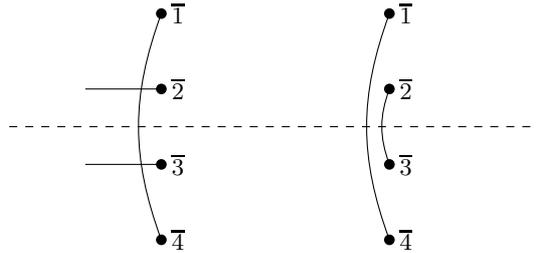
\begin{center}
    \begin{tikzpicture}
\fill (0,0) circle[radius=2pt];
\fill (0,1) circle[radius=2pt];
\fill (0,2) circle[radius=2pt];
\fill (0,3) circle[radius=2pt];
\draw (0,0) node[right] {\footnotesize $\ol{4}$};
\draw (0,1) node[right] {\footnotesize $\ol{3}$};
\draw (0,2) node[right] {\footnotesize $\ol{2}$};
\draw (0,3) node[right] {\footnotesize $\ol{1}$};
\draw (0,1) -- (-1,1);
\draw (0,2) -- (-1,2);
\path[-] (0,0) edge [bend left=20] (0,3);

\draw[dashed] (-2,1.5) -- (5,1.5);

\fill (3,0) circle[radius=2pt];
\fill (3,1) circle[radius=2pt];
\fill (3,2) circle[radius=2pt];
\fill (3,3) circle[radius=2pt];
\draw (3,0) node[right] {\footnotesize $\ol{4}$};
\draw (3,1) node[right] {\footnotesize $\ol{3}$};
\draw (3,2) node[right] {\footnotesize $\ol{2}$};
\draw (3,3) node[right] {\footnotesize $\ol{1}$};
\path[-] (3,0) edge [bend left=20] (3,3);
\path[-] (3,1) edge [bend left=20] (3,2);
\end{tikzpicture}
\end{center}
\end{eg}

\begin{defn}
Consider a right link state $p \in P_i$. Let $J_p$ denote the left ideal of $\mathcal{RB}_n(\delta, \varepsilon)$ with basis given by the diagrams having right link state obtained from $p$ by a (possibly empty) sequence of splices and deletions.
\end{defn}

One can see that $J_p$ is a left ideal as follows: given rook-Brauer $n$-diagrams $x$ and $y$, the right link state of the composite $xy$ is obtained from the right link state of $y$ by means of a valid sequence of splicings and deletions.

\subsection{Sesqui-diagrams and double diagrams}
\label{sesqui-subsec}

In order to prove our results it will be useful to consider the diagrams formed in composition, together with the composite of a diagram and a right link state. This leads to the notions of \emph{double diagram} and \emph{sesqui-diagram}.

\begin{defn}
For any of the algebras defined in Sections \ref{alg-sec}, the composition of two diagrams $d_1$ and $d_2$ involved forming a diagram $d_1\ast d_2$ with three columns of vertices, formed by identifying the right-hand column of vertices of $d_1$ with the left-hand column of vertices in $d_2$. We call such a diagram a \emph{double diagram}. Henceforth, when using double diagrams we will label the left-hand column of vertices by $1,\dotsc , n$, the middle column of vertices with $1^{\prime},\dotsc , n^{\prime}$ and the right-hand column of vertices with $\ol{1},\dotsc , \ol{n}$.     
\end{defn}

\begin{eg}
The diagram
\begin{center}
    \begin{tikzpicture}
\fill (0,0) circle[radius=2pt];
\fill (0,1) circle[radius=2pt];
\fill (0,2) circle[radius=2pt];
\fill (0,3) circle[radius=2pt];
\fill (1,0) circle[radius=2pt];
\fill (1,1) circle[radius=2pt];
\fill (1,2) circle[radius=2pt];
\fill (1,3) circle[radius=2pt];
\fill (2,0) circle[radius=2pt];
\fill (2,1) circle[radius=2pt];
\fill (2,2) circle[radius=2pt];
\fill (2,3) circle[radius=2pt];
\path[-] (0,0) edge [bend right=20] (0,2);
\path[-] (1,2) edge [bend right=20] (1,3);
\path[-] (1,2) edge [bend left=20] (1,3);
\draw (0,3) -- (1,0) -- (2,2);
\path[-] (2,1) edge [bend left=20] (2,3);
\end{tikzpicture}
\end{center}
from Example \ref{comp-eg} is a double diagram. 
\end{eg}

\begin{defn}
For any of the algebras defined in Section \ref{alg-sec}, let $p$ be a right link state of an $n$-diagram. Let $d$ be an $n$-diagram. We define the \emph{sesqui-diagram} $(p,d)$ to be the diagram formed by identifying the vertices $\ol{1},\dotsc , \ol{n}$ in $p$ with the vertices $1,\dotsc , n$ in $d$. After performing this identification, we relabel the vertices. The left-hand column of vertices will be labelled by $1^{\prime},\dotsc ,n^{\prime}$ and we preserve the labels $\ol{1},\dotsc , \ol{n}$ on the right-hand column.  
\end{defn}

\begin{eg}
Consider the right link state $p$ and the diagram $d$ (as walled Brauer $(2+2)$-diagrams) below:
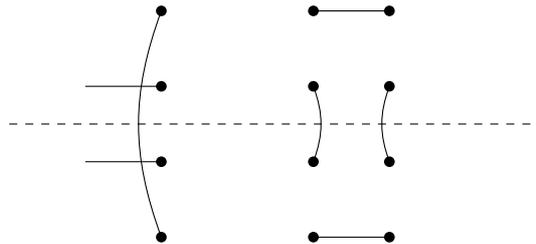
\begin{center}
    \begin{tikzpicture}
\fill (0,0) circle[radius=2pt];
\fill (0,1) circle[radius=2pt];
\fill (0,2) circle[radius=2pt];
\fill (0,3) circle[radius=2pt];
\draw (0,1) -- (-1,1);
\draw (0,2) -- (-1,2);
\path[-] (0,0) edge [bend left=20] (0,3);

\draw[dashed] (-2,1.5) -- (5,1.5);

\fill (2,0) circle[radius=2pt];
\fill (2,1) circle[radius=2pt];
\fill (2,2) circle[radius=2pt];
\fill (2,3) circle[radius=2pt];
\fill (3,0) circle[radius=2pt];
\fill (3,1) circle[radius=2pt];
\fill (3,2) circle[radius=2pt];
\fill (3,3) circle[radius=2pt];
\draw (2,3) -- (3,3);
\draw (2,0) -- (3,0);
\path[-] (2,1) edge [bend right=20] (2,2);
\path[-] (3,1) edge [bend left=20] (3,2);

\end{tikzpicture}
\end{center}

The sesqui-diagram $(p,d)$ is drawn below:
\begin{center}
    \begin{tikzpicture}
\fill (0,0) circle[radius=2pt];
\fill (0,1) circle[radius=2pt];
\fill (0,2) circle[radius=2pt];
\fill (0,3) circle[radius=2pt];
\draw (0,1) -- (-1,1);
\draw (0,2) -- (-1,2);
\path[-] (0,0) edge [bend left=20] (0,3);

\draw[dashed] (-2,1.5) -- (3,1.5);

\fill (0,0) circle[radius=2pt];
\fill (0,1) circle[radius=2pt];
\fill (0,2) circle[radius=2pt];
\fill (0,3) circle[radius=2pt];
\fill (1,0) circle[radius=2pt];
\fill (1,1) circle[radius=2pt];
\fill (1,2) circle[radius=2pt];
\fill (1,3) circle[radius=2pt];
\path[-] (1,1) edge [bend left=20] (1,2);
\draw (0,0) -- (1,0);
\path[-] (0,1) edge [bend right=20] (0,2);
\draw (0,3) -- (1,3);
\end{tikzpicture}
\end{center}
\end{eg}

\section{Cohomology of diagram algebras}
\label{Boyde-sec}
In this section we prove the cohomological analogues of \cite[Theorem 4.3]{Boyde} and \cite[Corollary 4.4]{Boyde} and provide an application in the case of subalgebras of the rook-Brauer algebras.

\begin{thm}
\label{Ext-thm}
Let $A$ be an associative $k$-algebra. Let $M$ and $N$ be right $A$-modules. Suppose that $I$ is a two-sided ideal of $A$ such that
\begin{itemize}
\item $I$ acts trivially on $M$ and $N$ and
\item as a left ideal $I$ is isomorphic to a direct sum of left ideals $J_1\oplus \cdots \oplus J_p$ where each $J_i$ is generated as a left ideal by finitely many commuting idempotents.
\end{itemize} 
Then there is an isomorphism of graded $k$-modules
\[\Ext_A^{\star}(M,N) \cong \Ext_{A/I}^{\star}(M,N).\]
\end{thm}
\begin{proof}
We will show that $\Ext_A^{\star}(M,N)$ and $\Ext_{A/I}^{\star}(M,N)$ are the cohomology of the same cochain complex.

Let $F_{\star}$ be a free resolution of $M$ by right $A$-modules. We know that $\Ext_A^{\star}(M,N)$ is the cohomology of the cochain complex $\Hom_A(F_{\star},N)$.

On the other hand, since $I$ acts trivially on $N$, we have an isomorphism of cochain complexes
\[\Hom_A(F_{\star},N) \cong \Hom_{A/I}(F_{\star}\otimes_A (A/I) , N),\]
by extension and restriction of scalars.

As demonstrated in the proof of \cite[Theorem 4.3]{Boyde}, $F_{\star}\otimes_A (A/I)$ is a resolution of $M$ by free right $(A/I)$-modules. This uses the commuting idempotents in the statement and the fact that $I$ acts trivially on $M$. 

As shown in \cite[Lemma 4.1]{Boyde}, if $e \in A$ is idempotent then the left ideal $Ae$ is a projective left $A$-module because there is a split short exact sequence
$$
0 \xrightarrow{} Ae \xrightarrow{} A \xrightarrow{\cdot (1-e)} A(1-e) \xrightarrow{} 0.
$$
Then \cite[Corollary 4.2]{Boyde} shows by induction that a left ideal of $A$ generated by finitely many commuting idempotents is in fact generated by a single idempotent, and hence is a projective left $A$-module.

Each $J_i$ is then a projective left $A$-module since it is generated as a left ideal of $A$ by finitely many commuting idempotents, and therefore $I$ is a projective left $A$-module as it is a direct sum of projective left $A$-modules. This yields a projective resolution $0\rightarrow I \hookrightarrow A \rightarrow 0$ of $A/I$ by projective right $A$-modules. 

The fact that $I$ acts trivially on $M$ allows one to deduce that when we tensor over $A$ with $M$ we obtain $M\otimes_A I=0$. 

Combining these two things we see that $\Tor_{\star}^A(M,A/I)=0$ for $\star\geqslant 1$ and $\Tor_{0}^A(M,A/I)=M$ as required. Therefore the cohomology of the right-hand side, $\Ext_{A/I}^{\star}(M,N)$, is isomorphic to $\Ext_A^{\star}(M,N)$ as required.
\end{proof}

The theorem has two important corollaries. The first covers the case where $I$ itself is generated as a left ideal by finitely many commuting idempotents. This follows immediately from the theorem and we state it without proof. The second covers the case where we have a chain of inclusions of two-sided ideals and each successive quotient of ideals satisfies the conditions of the theorem.

\begin{cor}
\label{Ext-cor-1}
Let $A$ be an associative $k$-algebra. Let $M$ and $N$ be right $A$-modules. Suppose that $I$ is a two-sided ideal of $A$ such that $I$ acts trivially on $M$ and $N$ and that $I$ is generated as a left ideal by finitely many commuting idempotents. Then there is an isomorphism of graded $k$-modules
\[\Ext_A^{\star}(M,N) \cong \Ext_{A/I}^{\star}(M,N).\]    
\end{cor}

\begin{cor}
\label{Ext-cor}
Let $A$ be a unital, associative $k$-algebra. Let $M$ and $N$ be right $A$-modules. Let $0\leqslant l \leqslant m$. Suppose that we have a chain of two-sided ideals
\[0=I_{-1}\leqslant I_0\leqslant I_1 \leqslant \cdots \leqslant I_m\leqslant A,\]
such that
\begin{itemize}
\item each $I_j$ acts trivially on $M$ and $N$;
\item for $i\geqslant l$ there is an isomorphism of left $A$-modules
\[\frac{I_i}{I_{i-1}}\cong \frac{J_{i,1}}{I_{i-1}}\oplus \cdots \oplus \frac{J_{i,p_i}}{I_{i-1}},\]
for left ideals $J_{i,j}$ generated by finitely many commuting idempotents.
\end{itemize}
Then there exist isomorphisms of graded $k$-modules
\[\Ext_{A/I_{l-1}}^{\star}(M,N) \cong \Ext_{A/I_{l}}^{\star}(M,N) \cong \cdots \cong \Ext_{A/I_{m}}^{\star}(M,N).\]
\end{cor}
\begin{proof}
Starting with the given chain of two-sided ideals and $l\geqslant 0$, the conditions in the statement tell us that for each $i\geqslant l$, the conditions of Theorem \ref{Ext-thm} are satisfied for the image of each two-sided ideal $I_i$ inside the algebra $A/I_{i-1}$. Theorem \ref{Ext-thm} therefore gives the required chain of isomorphisms.
\end{proof}

\begin{thm}
\label{technical-thm}
Let $0\leqslant l\leqslant m\leqslant n-1$. Let $A$ be a subalgebra of $\mathcal{RB}_n(\delta , \varepsilon)$ such that
\begin{itemize}
\item as a $k$-module, $A$ is free on a subset of the rook-Brauer diagrams and
\item for $i$ in the range $l\leqslant i \leqslant m$, for each right link state $p\in P_i$,  if $A$ contains at least one diagram with right link state $p$, then $A$ contains an idempotent $e_p$ such that in $A$ we have an equality of left ideals $A\cdot e_p = A\cap J_p$.
\end{itemize}
Then we have a chain of isomorphisms 
\[\Ext_{A/\left(A\cap I_{l-1}\right)}^{\star}\left(\mathbbm{1} , \mathbbm{1}\right) \cong \Ext_{A/\left(A\cap I_{l}\right)}^{\star}\left(\mathbbm{1} , \mathbbm{1}\right)\cong \cdots \cong \Ext_{A/\left(A\cap I_{m}\right)}^{\star}\left(\mathbbm{1} , \mathbbm{1}\right) .\]
\end{thm}
\begin{proof}
The proof of this theorem follows the same method as \cite[Theorem 1.11]{Boyde}, given in Section 6 of that paper. We will apply Corollary \ref{Ext-cor} to the chain of two-sided ideals 
\[0=A\cap I_{-1} \leqslant A \cap I_0 \leqslant \cdots \leqslant A \cap I_m \leqslant A.\]

We note that each of these two-sided ideals acts as multiplication by $0\in k$ on $\1$ since the diagrams spanning these ideals have fewer than $n$ propagating edges. Therefore the first condition of Corollary \ref{Ext-cor} is satisfied.

For the second condition, since each $I$ is free as a $k$-module on some basis of diagrams and since every diagram has some right link state, we can write
\[\frac{A \cap I_{i}}{A\cap I_{i-1}} = \sum_{p\in P_i} \frac{A \cap J_p}{A\cap I_{i-1}},\]
where we recall that the set $P_i$ consists of the right link states of the diagrams that span the quotient $I_i/I_{i-1}$. The fact that the sum is direct follows by noting that, if we have two distinct elements in $P_i$, say $p$ and $q$, then $J_p\cap J_q\subset I_{i-1}$. Since $p$ and $q$ are distinct, an element in $J_p\cap J_q$ must be obtained from both $p$ and $q$ by at least one splice or deletion (depending on which operations are allowed in $A$). In either case, this yields a diagram in $I_{i-1}$ since the number of defects is reduced. 
\end{proof}

\section{Cohomological analogues of known results}
\label{known-results-sec}

In this section, we collect some cohomological analogues of known results on the homology of Brauer algebras, Temperley--Lieb algebras and rook algebras. 

\begin{prop}
\label{Brauer-prop}
If $\delta \in k$ is invertible or if $n$ is odd, then there is an isomorphism of graded $k$-modules $\Ext_{\mathcal{B}_n(\delta)}^{\star}\left(\mathbbm{1},\mathbbm{1}\right) \cong H^{\star}\left(\Sigma_n , \1\right)$.
\end{prop}
\begin{proof}
    We apply Theorem \ref{technical-thm} with $l=0$ and $m=n-1$ to obtain this result. The chain of isomorphisms tells us that the cohomology of 
    \[\mathcal{B}_n(\delta)/(\mathcal{B}_n(\delta)\cap I_{-1})=\mathcal{B}_n(\delta)\]
    is isomorphic to the cohomology of \[\mathcal{B}_n(\delta)/(\mathcal{B}_n(\delta)\cap I_{n-1})\cong k[\Sigma_n],\]
    whence the result. The Brauer algebra is free on the basis of Brauer diagrams. The existence of the necessary idempotents was shown by Boyde: \cite[Lemma 7.1]{Boyde} constructs the necessary idempotent when $\delta$ is invertible; \cite[Section 8]{Boyde} describes how to construct the necessary idempotent when $n$ is odd. The latter case uses the notions of double diagram and sesqui-diagram recalled in Subsection \ref{sesqui-subsec} and we adapt these methods to study the walled Brauer algebras in Section \ref{WB-sec}. 
\end{proof}

\begin{prop}
If $n$ is odd, then the graded $k$-module $\Ext_{\mathcal{TL}_n(\delta)}^{\star}\left(\mathbbm{1},\mathbbm{1}\right)$ is isomorphic to $k$ concentrated in degree zero.
\end{prop}
\begin{proof}
    The proof is similar to Proposition \ref{Brauer-prop}. We apply Theorem \ref{technical-thm}, except in this case the chain of isomorphisms tells us that the cohomology of $\mathcal{TL}_{n}(\delta)/(\mathcal{TL}_n(\delta)\cap I_{-1})=\mathcal{TL}_{n}(\delta)$ is isomorphic to the cohomology of $\mathcal{TL}_{n}(\delta)/(\mathcal{TL}_n(\delta)\cap I_{n-1})=k$. We can apply the theorem because the Temperley--Lieb algebras are free on the basis of Temperley--Lieb algebras and Boyde has constructed the necessary idempotents in \cite[Section 8]{Boyde}.
\end{proof}

\begin{prop}
Let $\varepsilon \in k$ be invertible. There exists an isomorphism of graded $k$-modules \[\Ext_{\mathcal{R}_n(\varepsilon)}^{\star}\left(\mathbbm{1},\mathbbm{1}\right) \cong H^{\star}\left(\Sigma_n , \1\right).\]
Furthermore, both the graded $k$-modules $\Tor_{\star}^{\mathcal{PR}_n(\varepsilon)}\left(\mathbbm{1},\mathbbm{1}\right)$ and $\Ext_{\mathcal{PR}_n(\varepsilon)}^{\star}\left(\mathbbm{1},\mathbbm{1}\right)$ are isomorphic to a copy of $k$ concentrated in degree zero.
\end{prop}
\begin{proof}
    In the introduction we recalled the diagrams $\rho_i$, defined in \cite[Section 5]{Boyde}. The diagram $\rho_i$ is formed by taking the rook-Brauer $n$-diagram with $n$ horizontal edges and omitting the $i^{\mathrm{th}}$ edge. In particular, each $\rho_i$ lies in the rook algebra $\mathcal{R}_n(\varepsilon)$ and in the planar rook algebra $\mathcal{PR}_n(\varepsilon)$.  
    
    Combining Lemmas 5.2 and 5.3 with the proof of Theorem 5.4 in \cite{Boyde}, we know that the elements $\varepsilon^{-1}\rho_i$ are idempotent and that they generate the two-sided ideal $\mathcal{R}_n(\varepsilon)\cap I_{n-1}$ as a left ideal. The statement for the rook algebras now follows from Corollary \ref{Ext-cor-1} by noting that $\mathcal{R}_n(\varepsilon)/(\mathcal{R}_n(\varepsilon)\cap I_{n-1}) \cong k[\Sigma_n]$.

    The proof for the planar rook algebras is similar. Using Boyde's methods one checks that the two-sided ideal $\mathcal{PR}_n(\varepsilon)\cap I_{n-1}$ is generated as a left ideal by the elements $\varepsilon^{-1}\rho_i$. We then apply Corollary \ref{Ext-cor-1}, noting that \[\mathcal{PR}_n(\varepsilon)/(\mathcal{PR}_n(\varepsilon)\cap I_{n-1}) \cong k\]
    since the only permutation diagram in $\mathcal{PR}_n(\varepsilon)$ is the identity diagram.
\end{proof}

\section{Cohomology of rook-Brauer algebras and Motzkin algebras}
\label{RB-M-sec}

In this section, we prove that if $\varepsilon$ is invertible then the (co)homology of the rook-Brauer algebra $\mathcal{RB}_n(\delta,\varepsilon)$ is isomorphic to the (co)homology of the symmetric group $\Sigma_n$ and that the (co)homology of the Motzkin algebras vanishes in positive degrees.

\begin{defn}
Let $p\in P_i$ be a right link state of a rook-Brauer $n$-diagram. Let $d_p$ be the rook-Brauer $n$-diagram defined as follows:
\begin{itemize}
    \item if $\ol{i}$ is an isolated vertex in $p$, then $i$ and $\ol{i}$ are isolated in $d_p$;
    \item if $\ol{i}$ and $\ol{j}$ are connected by a non-propagating edge in $p$, then they are also connected by a non-propagating edge in $d_p$;
    \item if there is a defect at $\ol{i}$ in $p$ then there is a propagating edge between $i$ and $\ol{i}$ in $d_p$;
    \item there are no non-propagating edges in the left-hand column of $d_p$.
\end{itemize}
\end{defn}

\begin{lem}
\label{RB-lem}
Let $p\in P_i$ be a right link state of a rook-Brauer $n$-diagram. Let $\alpha$ be the number of isolated vertices in $p$ and let $\beta$ be the number of non-propagating edges in $p$. Then, for any $y\in J_p$, we have $yd_p=\varepsilon^{\alpha+\beta}y$.    
\end{lem}
\begin{proof}
We will check the following equivalent statement: the double diagram $y\ast d_p$ has all the vertex pairings and isolated vertices of $y$ with no loops and precisely $\alpha+\beta$ contractible components in the middle column.

Any non-propagating edge or isolated vertex in the left-hand column of $y$ appears in the left-hand column of $y\ast d_p$.

Suppose $y$ has a propagating edge from $i$ to $\ol{j}$. Therefore the double diagram $y\ast d_p$ has an edge from $i$ to $j^{\prime}$. A propagating edge in $y$ leads to a defect in its right link state at $\ol{j}$. Since $y\in J_p$, its right link state is obtained from $p$ by a sequence of splices and deletions. Therefore $p$ must have a defect at $\ol{j}$. By definition of $d_p$, there is an edge from $j^{\prime}$ to $\ol{j}$ in the double diagram $y\ast d_p$. In other words $i$ and $\ol{j}$ are connected in the double diagram as required.

Suppose $y$ has isolated vertices in the right-hand column. There are precisely $\alpha$ of these that are isolated in $p$, giving a factor of $\varepsilon^{\alpha}$ in the product. Any other isolated vertex in the right-hand column of $y$ is the result of a deletion of a defect in its right link state. By definition of $d_p$, there is an edge joining such an isolated vertex to the right-hand column in the double diagram. This leads to the vertex being isolated in the product without generating any extra factors of $\varepsilon$.

Suppose $y$ has a non-propagating edge in its right-hand column. There are precisely $\beta$ of these that appear in $p$, giving a factor of $\varepsilon^{\beta}$ in the product. Any other non-propagating edge in the right-hand column of $y$ is the result of a splice in its right link state. By definition of $d_p$, there are edges joining the two vertices to the right-hand column in the double diagram. This leads to the non-propagating edge existing in the right-hand column of the product.

Finally, in order to form a loop, we would need non-propagating edges in the left-hand column of $d_p$ but, by definition, there are none.

Therefore, the double diagram $y\ast d_p$ has all the vertex pairings and isolated vertices of $y$ with no loops and precisely $\alpha+\beta$ contractible components in the middle column, as required.
\end{proof}

\begin{thm}
\label{RB-thm}
Let $\varepsilon$ be invertible. For any $\delta$, there exist isomorphisms of graded $k$-modules
\[\Tor_{\star}^{\mathcal{RB}_n(\delta,\varepsilon)}(\1 ,\1) \cong H_{\star}(\Sigma_n ,\1) \quad \text{and} \quad \Ext_{\mathcal{RB}_n(\delta,\varepsilon)}^{\star}(\1 ,\1) \cong H^{\star}(\Sigma_n ,\1).\]
\end{thm}
\begin{proof}
We apply Theorem \ref{technical-thm} with $l=0$ and $m=n-1$. As a $k$-module, the rook-Brauer algebra is free on a subset of rook-Brauer diagrams (all of them, in fact).

Let $e_p=\varepsilon^{-(\alpha+\beta)}d_p$. Lemma \ref{RB-lem} implies that $e_p$ is idempotent. It also implies that for any $y\in J_p$, $ye_p= \varepsilon^{-(\alpha+\beta)}yd_p=y$ and so $J_p\subseteq A\cdot e_p$. The reverse inclusion, $A\cdot e_p\subseteq J_p$, holds since $e_p$ has right link state $p$ by construction. 
\end{proof}

\begin{thm}
\label{Motzkin-thm}
Let $\varepsilon$ be invertible. Then for any $\delta$, the graded $k$-modules $\Tor_{\star}^{\mathcal{M}_n(\delta,\varepsilon)}(\1 ,\1)$ and $\Ext_{\mathcal{M}_n(\delta,\varepsilon)}^{\star}(\1 ,\1)$ are isomorphic to $k$ concentrated in degree zero.
\end{thm}
\begin{proof}
The method for the Motzkin algebras is exactly the same as for the rook-Brauer algebras. The key point to note is that if we start with a right link state of a Motzkin diagram, then the diagram $d_p$ will be a Motzkin diagram.    
\end{proof}

\section{Cohomology of walled Brauer algebras}
\label{WB-sec}

We start by showing that the (co)homology of the walled Brauer algebra $\mathcal{B}_{r,s}(\delta)$ is isomorphic to the cohomology of $\Sigma_r\times \Sigma_s$ when the parameter $\delta$ is invertible. We recall the definition of $k$-free idempotent left cover together with a related result on (co)homology. We then construct a $k$-free idempotent left cover of the two-sided ideal $I_{r+s-1}$. We use this to show that the (co)homology of the walled Brauer algebra $\mathcal{B}_{r,s}(\delta)$ is isomorphic to the (co)homology of $\Sigma_r\times \Sigma_s$ when $r\neq s$ for any $\delta \in k$ and that such an isomorphism holds in a range when $r=s$. We show that this range is sharp in the case $r=s=1$.

\subsection{Invertible parameter}

We begin with the case where the parameter $\delta$ is invertible.

\begin{defn}
\label{dp-defn}
Fix $p\in P_i$, a right link state in $\mathcal{B}_{r,s}(\delta)$ with precisely $i$ defects. Define a diagram $d_p\in \mathcal{B}_{r,s}(\delta)$ as follows:
\begin{itemize}
    \item if there is a non-propagating edge between $\ol{a}$ and $\ol{b}$ in $p$ then there is a non-propagating edge between $a$ and $b$ and a non-propagating edge between $\ol{a}$ and $\ol{b}$ in $d_p$;
    \item if there is a defect at vertex $\ol{a}$ in $p$, then there is a propagating edge between $a$ and $\ol{a}$ in $d_p$.
\end{itemize}
\end{defn}

\begin{lem}
\label{WB-diag-lem}
Fix $p\in P_i$, a right link state in $\mathcal{B}_{r,s}(\delta)$ with precisely $i$ defects. If $y$ is a diagram in $\mathcal{B}_{r,s}(\delta) \cap J_p$, then $yd_p=\delta^{\alpha}y$, where $\alpha$ is the number of non-propagating edges in $p$.
\end{lem}
\begin{proof}
In order for $y$ to lie in $\mathcal{B}_{r,s}(\delta) \cap J_p$, the right link state of $y$ must be obtained from $p$ by a (possibly empty) sequence of splices. We note the following facts:
\begin{itemize}
    \item if $\ol{a}$ and $\ol{b}$ are connected by a non-propagating edge in $p$ then they are also connected by a non-propagating edge in $y$;
    \item if there is a defect at vertex $\ol{a}$ in $p$ then there are two possibilities:
    \begin{itemize}
        \item the vertex $\ol{a}$ is part of a propagating edge in $y$ (which results in a defect in its right link state);
        \item the vertex $\ol{a}$ is part of a non-propagating edge in the right-hand column of $y$ (which is the result of a splice).
    \end{itemize}
\end{itemize}
Now consider the double diagram $y\ast d_p$ (recalling the notation of Subsection \ref{sesqui-subsec}). We will show that if two vertices are connected in $y$ then they are connected in the double diagram $y\ast d_p$ and that there are precisely $\alpha$ loops in the middle column of $y\ast d_p$.

Any non-propagating edges in the left-hand column of $y$ are also present in the left-hand column of the double diagram $y\ast d_p$.

There are precisely $\alpha$ non-propagating edges that appear in both the right-hand column of $y$ and the right link state $p$. For each of these there is a corresponding non-propagating edge in the left-hand column of $d_p$ by construction. These form $\alpha$ loops in the middle column of $y\ast d_p$. Any other non-propagating edge in the right-hand column of $y$, say between the vertices $\ol{a}$ and $\ol{b}$, is the result of splicing two defects together. Since there were defects at $\ol{a}$ and $\ol{b}$ in the right link state $p$, there are edges from $a^{\prime}$ to $\ol{a}$ and from $b^{\prime}$ to $\ol{b}$ in the double diagram $y\ast d_p$. Combining these with the edge from $a^{\prime}$ to $b^{\prime}$ in the left-hand part of the double diagram $y\ast d_p$ formed by the splice, we see that $\ol{a}$ and $\ol{b}$ are connected in $y\ast d_p$.

Finally, if there is a propagating edge from $a$ to $\ol{b}$ in $y$ then there is an edge from $a$ to $b^{\prime}$ in the double diagram $y\ast d_p$. Since this propagating edge would lead to a defect in the right link state of $y$, there is an edge from $b^{\prime}$ to $\ol{b}$ in $y\ast d_p$ by the construction of $d_p$.

Combining all of this, we see that all the vertex pairings in $y$ are present in $y\ast d_p$ with precisely $\alpha$ loops in the middle column. Therefore, $yd_p=\delta^{\alpha}y$. 
\end{proof}

\begin{thm}
\label{walled-Brauer-thm}
Let $r$ and $s$ be non-negative integers. Let $\delta \in k$ be invertible. There exist isomorphisms of graded $k$-modules
\[\Tor_{\star}^{\mathcal{B}_{r,s}(\delta)}\left(\mathbbm{1},\mathbbm{1}\right) \cong H_{\star}\left(\Sigma_r\times \Sigma_s , \1\right) \text{ and }  \Ext_{\mathcal{B}_{r,s}(\delta)}^{\star}\left(\mathbbm{1},\mathbbm{1}\right) \cong H^{\star}\left(\Sigma_r\times \Sigma_s , \1\right).\]
\end{thm}
\begin{proof}
The walled Brauer algebra $\mathcal{B}_{r,s}(\delta)$ is a subalgebra of the rook-Brauer algebra $\mathcal{RB}_{r+s}(\delta , \varepsilon)$.

We will apply \cite[Theorem 1.11]{Boyde} for the homological statement and Theorem \ref{technical-thm} for the cohomological statement with $l=0$ and $m=r+s-1$.

We note that 
\[\frac{\mathcal{B}_{r,s}(\delta)}{\left(\mathcal{B}_{r,s}(\delta) \cap I_{-1}\right)} = \mathcal{B}_{r,s}(\delta)\quad \text{and} \quad \frac{\mathcal{B}_{r,s}(\delta)}{\left(\mathcal{B}_{r,s}(\delta) \cap I_{r+s-1}\right)}\cong k[\Sigma_r\times \Sigma_s].\] The latter holds since the quotient is spanned $k$-linearly by walled Brauer $(r+s)$-diagrams having $r+s$ propagating edges. The conditions of walled Brauer diagrams tells us that such a diagram consists of a permutation of the first $r$ vertices and a permutation of the last $s$ vertices.

The first condition of the theorem is satisfied since the walled Brauer algebra $\mathcal{B}_{r,s}(\delta)$ is free on a subset of rook-Brauer $(r+s)$-diagrams, namely the walled Brauer $(r+s)$-diagrams.

Let $e_p=\delta^{-\alpha}d_p$. Lemma \ref{WB-diag-lem}, with $y=d_p$, implies that $e_p$ is idempotent. There is an inclusion $\mathcal{B}_{r,s}(\delta)\cdot e_p \subset \mathcal{B}_{r,s}(\delta) \cap J_p$ since $e_p$ lies in $J_p$. For the reverse inclusion, take $y\in \mathcal{B}_{r,s}(\delta) \cap J_p$. Lemma \ref{WB-diag-lem} tells us that $ye_p=\delta^{-\alpha}yd_p=y$, so $y\in \mathcal{B}_{r,s}(\delta)\cdot e_p$ as required.
\end{proof}

\subsection{A result on (co)homology}
We recall the concept of $k$-free idempotent cover from \cite[Section 2]{Boyde2}, together with a result on (co)homology.

\begin{defn}
Let $A$ be a $k$-algebra. Let $I$ be a two-sided ideal of $A$. Let $w\geqslant h \geqslant 1$. An \emph{idempotent left cover of $I$ of height $h$ and width $w$} is a collection of left ideals $J_1,\dotsc , J_w$ in $A$ such that
\begin{itemize}
    \item $J_1+\cdots +J_w=I$;
    \item for $S\subset \llb 1,2,\dotsc ,w\rrb$ with $\llv S\rrv \leqslant h$, the intersection
    \[\bigcap_{i\in S} J_i\]
    is either zero or is a principal left ideal generated by an idempotent.
\end{itemize}
If $I$ is free as a $k$-module, then an idempotent left cover is said to be \emph{$k$-free} if there is a choice of $k$-basis for $I$ such that each $J_i$ is free on a subset of this basis.
\end{defn}

We will make use of the following result which can found as \cite[Theorem B]{Fisher-Graves} (noting that the homological part was originally published as \cite[Theorem 2.7]{Boyde2}).

\begin{prop}
\label{stable-isos-prop}
Let $A$ be an augmented $k$-algebra with trivial module $\mathbbm{1}$. Let $I$ be a two-sided ideal of $A$ which is free as a $k$-module and which acts as multiplication by $0\in k$ on $\mathbbm{1}$. Suppose that there exists a $k$-free idempotent left cover of $I$ of height $h$ and width $w$. There are natural isomorphisms of $k$-modules
\[\Tor_q^A(\mathbbm{1},\mathbbm{1}) \cong \Tor_q^{A/I}(\mathbbm{1},\mathbbm{1})
\quad
\text{and} \quad
\Ext_A^q(\mathbbm{1},\mathbbm{1}) \cong \Ext_{A/I}^q(\mathbbm{1},\mathbbm{1})\]
for $q\leqslant h$. Furthermore, if $h=w$, then the isomorphisms holds for all $q$.
\end{prop}

\subsection{Technical lemmas}

Recall the notions of double diagram and sesqui-diagram from Subsection \ref{sesqui-subsec}, together with the notation introduced therein.

\begin{lem}
\label{WB-link-state-lem}
Let $p$ be the right link state of a walled Brauer $(r+s)$-diagram. Suppose that there exists a walled Brauer $(r+s)$-diagram $e$ such that
\begin{enumerate}
    \item\label{WB1} $e$ has right link state $p$;
    \item\label{WB2} if $p$ has a defect at vertex $\ol{b}$ then there are sequence of edges in the sesqui-diagram $(p,e)$ that connects $b^{\prime}$ and $\ol{b}$;
    \item\label{WB3} each vertex $b^{\prime}$ in $(p,e)$ is connected to the right-hand column of vertices.
\end{enumerate}
Then $ye=y$ for all $y\in \mathcal{B}_{r,s}(\delta)\cap J_p$.
\end{lem}
\begin{proof}
We will show that if two vertices are connected in the diagram $y$ then they are connected in the double diagram $y\ast e$ and that the double diagram $y\ast e$ contains no loops. If there is a non-propagating edge between $a$ and $b$ in the left-hand column of $y$ then there must also be a non-propagating edge between $a$ and $b$ in the left-hand column of the double diagram $y\ast e$. 

If $y$ has a propagating edge from $a$ to $\ol{b}$, then there is an edge from $a$ to $b^{\prime}$ in the double diagram $y\ast e$. Condition \ref{WB2} tells us that $b^{\prime}$ is connected to $\ol{b}$ in the double diagram $y\ast e$ as required.

Suppose $y$ contains a non-propagating edge between $\ol{a}$ and $\ol{b}$ in its right-hand column. If this non-propagating edge is present in the right link state $p$ then $\ol{a}$ and $\ol{b}$ are connected in the double diagram $y\ast e$ by Condition \ref{WB1}. If the non-propagating edge is not present in the right link state $p$, it is the result of splicing two defects together. In this case $a^{\prime}$ and $b^{\prime}$ are connected in the double diagram $y\ast e$. Condition \ref{WB2} tells us that $a^{\prime}$ is connected to $\ol{a}$ and that $b^{\prime}$ is connected to $\ol{b}$ and so $\ol{a}$ and $\ol{b}$ are connected in the double diagram $y\ast e$ as required.

Finally, Condition \ref{WB3} tells us that every vertex in the middle column of the double diagram $y\ast e$ is connected to the right-hand column and therefore $y\ast e$ cannot contain any loops.
 \end{proof}

\begin{lem}
\label{WB-idem-lem}
Let $p$ be the right link state of a walled Brauer $(r+s)$-diagram with at least one defect. There exists a diagram $e_p$ in $\mathcal{B}_{r,s}(\delta)$ satisfying the conditions of Lemma \ref{WB-link-state-lem}.
\end{lem}
\begin{proof}
We start with two copies of the right link state $p$ placed side-by-side. We will explain how to extend the right-hand copy of $p$ to the necessary diagram $e_p$. We note that Condition \ref{WB1} of Lemma \ref{WB-link-state-lem} is automatically satisfied since we are starting with $p$ as the right link state. There is an illustrative example demonstrating the steps of this lemma following this proof.

Label the vertices of the left-hand copy by $1, 2, \dotsc , r+s$ and label the vertices of the right-hand copy by $\ol{1},\ol{2},\dotsc , \ol{r+s}$.

Extend all but one of the defects in the right-hand copy of $p$ to be connected to the corresponding defect in the left-hand copy of $p$ by a horizontal edge (since these propagating edges are horizontal, they satisfy the conditions on edges in a walled Brauer diagram). We note that the choice of which defect to leave unconnected is immaterial, the following method applies equally well for any choice.

We will join the final defect in the right-hand column of $p$ to the corresponding defect in the left-hand column of $p$, but via a sequence of edges that passes through every non-propagating edge in the left-hand copy of $p$. We note that when we have done this, each defect in the right-hand copy of $p$ will be connected to the corresponding defect in the left-hand copy (so Condition \ref{WB2} of Lemma \ref{WB-link-state-lem} will be satisfied) and every vertex in the left-hand copy of $p$ will be connected to a vertex in the right-hand copy of $p$ (so Condition \ref{WB3} of Lemma \ref{WB-link-state-lem} will be satisfied). We proceed as in \cite[Lemma 8.7]{Boyde}, but we must take extra care to make sure that the non-propagating edges that we add to form our path between the two remaining defects satisfy the conditions of a walled Brauer diagram.

Suppose the right link state $p$ has $x$ non-propagating edges. By the conditions on walled Brauer diagrams, there are $x$ vertices amongst the first $r$ vertices and $x$ vertices amongst the last $s$ vertices that are part of non-propagating edges. Choose a total ordering $1<\cdots <x$ on the non-propagating edges in the left-hand copy of $p$ (any total ordering will do). Label the vertices corresponding to the total ordering $r_1,\dotsc , r_x$ for the vertices in the top part of the diagram and $s_1,\dotsc ,s_x$ for the vertices in the bottom part of the diagram.

Suppose that our final unconnected defect in the right-hand copy of $p$ is amongst the last $s$ vertices, at vertex $\ol{a}$. Connect our last remaining defect in the right-hand copy of $p$ to the vertex $s_1$ in the left-hand copy of $p$. Now join vertex $r_i$ to $s_{i+1}$ (for $1\leqslant i\leqslant x-1$) by non-propagating edges to the right of the vertices in the left-hand copy of $p$. Finally, join $r_x$ to the vertex $a$.

We note that had our remaining defect been amongst the first $r$ vertices (still at some vertex $\ol{a}$) we would join the defect to vertex $r_1$, join $s_i$ to $r_{i+1}$ for $1\leqslant i \leqslant x-1$ and then join $s_x$ to vertex $a$.
\end{proof}

\begin{eg}
Suppose we start with the following right link state, $p$, of a walled Brauer $(2+3)$-diagram:

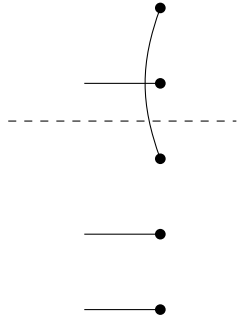
\begin{center}
    \begin{tikzpicture}
\fill (1,0) circle[radius=2pt];
\fill (1,1) circle[radius=2pt];
\fill (1,2) circle[radius=2pt];
\fill (1,3) circle[radius=2pt];
\fill (1,4) circle[radius=2pt];
\draw[dashed] (-1,2.5) -- (2,2.5);
\draw (1,0) -- (0,0);
\draw (1,1) -- (0,1);
\draw (1,3) -- (0,3);
\path[-] (1,2) edge [bend left=20] (1,4);
\end{tikzpicture}
\end{center}

We place two copies side by side and extend all but one of the defects in the right-hand copy to horizontal edges. In this picture, we have extended the two top-most defects.

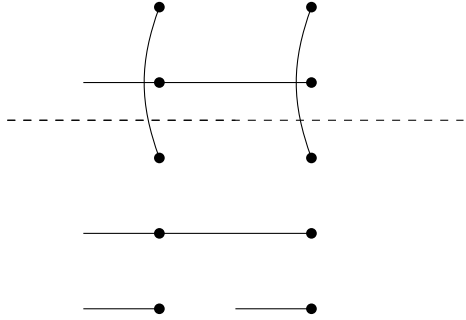
\begin{center}
    \begin{tikzpicture}
\fill (1,0) circle[radius=2pt];
\fill (1,1) circle[radius=2pt];
\fill (1,2) circle[radius=2pt];
\fill (1,3) circle[radius=2pt];
\fill (1,4) circle[radius=2pt];
\draw[dashed] (-1,2.5) -- (2,2.5);
\draw (1,0) -- (0,0);
\draw (1,1) -- (0,1);
\draw (1,3) -- (0,3);
\path[-] (1,2) edge [bend left=20] (1,4);

\fill (3,0) circle[radius=2pt];
\fill (3,1) circle[radius=2pt];
\fill (3,2) circle[radius=2pt];
\fill (3,3) circle[radius=2pt];
\fill (3,4) circle[radius=2pt];
\draw[dashed] (-1,2.5) -- (5,2.5);
\draw (3,0) -- (2,0);
\draw (3,1) -- (1,1);
\draw (3,3) -- (1,3);
\path[-] (3,2) edge [bend left=20] (3,4);
\end{tikzpicture}
\end{center}

Finally, we join the remaining defect in the right-hand copy to the corresponding defect in the left-hand copy via a path that passes through the non-propagating edge in the left-hand copy of the link state. We now have a sesqui-diagram $(p,e_p)$ where the diagram $e_p$ satisfies the conditions of Lemma \ref{WB-link-state-lem}. The diagram on the right is $e_p$.

\begin{center}
    \begin{tikzpicture}
\fill (1,0) circle[radius=2pt];
\fill (1,1) circle[radius=2pt];
\fill (1,2) circle[radius=2pt];
\fill (1,3) circle[radius=2pt];
\fill (1,4) circle[radius=2pt];
\draw[dashed] (-1,2.5) -- (2,2.5);
\draw (1,0) -- (0,0);
\draw (1,1) -- (0,1);
\draw (1,3) -- (0,3);
\path[-] (1,2) edge [bend left=20] (1,4);

\fill (3,0) circle[radius=2pt];
\fill (3,1) circle[radius=2pt];
\fill (3,2) circle[radius=2pt];
\fill (3,3) circle[radius=2pt];
\fill (3,4) circle[radius=2pt];
\draw[dashed] (-1,2.5) -- (5,2.5);
\draw (3,0) -- (2,0);
\draw (3,1) -- (1,1);
\draw (3,3) -- (1,3);
\path[-] (3,2) edge [bend left=20] (3,4);

\draw (2,0) -- (1,2);
\path[-] (1,4) edge [bend left=20] (1,0);

\fill (8,0) circle[radius=2pt];
\fill (8,1) circle[radius=2pt];
\fill (8,2) circle[radius=2pt];
\fill (8,3) circle[radius=2pt];
\fill (8,4) circle[radius=2pt];
\fill (9,0) circle[radius=2pt];
\fill (9,1) circle[radius=2pt];
\fill (9,2) circle[radius=2pt];
\fill (9,3) circle[radius=2pt];
\fill (9,4) circle[radius=2pt];
\draw[dashed] (6,2.5) -- (11,2.5);
\draw (9,0) -- (8,2);
\draw (9,1) -- (8,1);
\draw (9,3) -- (8,3);
\path[-] (9,2) edge [bend left=20] (9,4);
\path[-] (8,4) edge [bend left=20] (8,0);
\end{tikzpicture}
\end{center}
\end{eg}

\begin{lem}
\label{defect-lem}
Suppose $r\neq s$ (with $r,s\geqslant 1$). The right link state of any diagram in $\mathcal{B}_{r,s}(\delta)$ has at least one defect.
\end{lem}
\begin{proof}
Suppose $s>r$. Since any non-propagating edge in $\mathcal{B}_{r,s}(\delta)$ must connect an element from the first $r$ vertices to a vertex from the last $s$ vertices, there can be at most $r$ non-propagating edges and so we must have at least $s-r>0$ propagating edges. The same argument works for $r>s$.    
\end{proof}

\subsection{A free idempotent left cover}

\begin{defn}
Let $U$ denote the set of elements $(i,j)$ where $i$ and $j$ are positive integers satisfying $1\leqslant i\leqslant r$ and $r+1\leqslant j\leqslant r+s$.    
\end{defn}

Consider a walled Brauer $(r+s)$-diagram, $d$. Multiplying on the left by any other walled Brauer $(r+s)$-diagram preserves non-propagating edges in the right-hand column of $d$. We can therefore make the following definition.

\begin{defn}
For $(i,j)\in U$, let $L_{i,j}$ be the left ideal of $\mathcal{B}_{r,s}(\delta)$ spanned $k$-linearly by walled Brauer $(r+s)$-diagrams such that the vertices $\ol{i}$ and $\ol{j}$ are joined by a non-propagating edge.  
\end{defn}

\begin{lem}
\label{WB-cover-lem}
The collection of left ideals $L_{i,j}$ (for $(i,j)\in U$) cover the two-sided ideal $I_{r+s-1}$ in $\mathcal{B}_{r,s}(\delta)$.    
\end{lem}
\begin{proof}
A basis diagram in $L_{i,j}$ contains at least one non-propagating edge, and so cannot have $r+s$ propagating edges. Hence each $L_{i,j}$ is contained in $I_{r+s-1}$.

Conversely, a basis diagram in $I_{r+s-1}$ has fewer than $r+s$ propagating edges. Therefore, we must have at least one non-propagating edge in each column of vertices and so the basis diagram must lie in some $L_{i,j}$.
\end{proof}

\begin{lem}
\label{WB-zero-lem}
Let $T\subseteq U$. Then
\[ \bigcap_{(i,j)\in T} L_{i,j}=0\]
if and only if there exist distinct elements $a=(i_1,j_1)$ and $b=(i_2,j_2)$ in $T$ such that one of the co-ordinates of $a$ is equal to one of the co-ordinates of $b$. 
\end{lem}
\begin{proof}
If there exist distinct elements $a=(i_1,j_1)$ and $b=(i_2,j_2)$ in $T$ such that one of the co-ordinates of $a$ is equal to one of the co-ordinates of $b$, then a diagram in the intersection would have a vertex in the right-hand column with two distinct edges incident to it, which contradicts the definition of a walled Brauer $n$-diagram.

Conversely, suppose there are not distinct elements $a=(i_1,j_1)$ and $b=(i_2,j_2)$ in $T$ such that one of the co-ordinates of $a$ is equal to one of the co-ordinates of $b$. We construct a diagram in the intersection as follows. For each $(i,j)\in T$, we have a non-propagating edge joining $i$ and $j$ and a non-propagating edge joining $\ol{i}$ and $\ol{j}$. For each vertex $v$ not included in the indexing set $T$, we have a propagating edge joining $v$ and $\ol{v}$.
\end{proof}

\begin{lem}
\label{WB-odd-intersection-lem}
Suppose $r\neq s$ (with $r,s\geqslant 1$). Suppose there are not distinct elements $a=(i_1,j_1)$ and $b=(i_2,j_2)$ in $T$ such that one of the co-ordinates of $a$ is equal to one of the co-ordinates of $b$. Then the intersection
\[\bigcap_{(i,j)\in T} L_{i,j}\]
is principal and generated by an idempotent.
\end{lem}
\begin{proof}
Let $q$ be the right link state such that for each $(i,j)\in T$, we have a non-propagating edge joining $\ol{i}$ and $\ol{j}$ and for each vertex $v$ not included in the indexing set $T$ we have a defect. We have
\[\bigcap_{(i,j)\in T} L_{i,j} =J_q.\]
Since $r\neq s$, $q$ must have at least one defect by Lemma \ref{defect-lem}. Lemmas \ref{WB-link-state-lem} and \ref{WB-idem-lem} tell us that there exists a diagram $e_q$ such that right multiplication by $e_q$ gives a retraction $\mathcal{B}_{r,s}(\delta) \rightarrow \mathcal{B}_{r,s}(\delta) \cap J_q$. It follows from \cite[Lemma 2.6]{Boyde2} that $J_q$ is principal and generated by an idempotent.
\end{proof}

\begin{lem}
\label{WB-even-intersection-lem}
Suppose $r=s$ (with $r,s\geqslant 1$). Suppose there are not distinct elements $a=(i_1,j_1)$ and $b=(i_2,j_2)$ in $T$ such that one of the co-ordinates of $a$ is equal to one of the co-ordinates in $b$. Suppose that 
\[\bigcup_{(i,j)\in T} \lbrace i,j\rbrace  \neq \llb 1,2,\dotsc,r+s\rrb.\]
Then the intersection
\[\bigcap_{(i,j)\in T} L_{i,j}\]
is principal and generated by an idempotent.
\end{lem}
\begin{proof}
Since the union of all indices in $T$ does not include every vertex in the right-hand column,
the same proof as Lemma \ref{WB-odd-intersection-lem} applies here, since we must have at least one defect.
\end{proof}

\begin{prop}
\label{WB-odd-idem-cover-prop}
Suppose $r\neq s$ (with $r,s\geqslant 1$). The left ideals $L_{i,j}$ (for $(i,j)\in U$) form a $k$-free idempotent left cover of $I_{r+s-1}$ such that the height is equal to the width.   
\end{prop}
\begin{proof}
This follows from Lemmas \ref{WB-cover-lem}, \ref{WB-zero-lem} and \ref{WB-odd-intersection-lem}.    
\end{proof}

\begin{prop}
\label{WB-even-idem-cover-prop}
Suppose $r=s$ ($r,s\geqslant 1$). The left ideals $L_{i,j}$ (for $(i,j)\in U$) form a $k$-free idempotent left cover of $I_{r+s-1}$ of height $\frac{r+s}{2}-1$.   
\end{prop}
\begin{proof}
This follows from Lemmas \ref{WB-cover-lem}, \ref{WB-zero-lem} and \ref{WB-even-intersection-lem}.    
\end{proof}

\subsection{(Co)homology of walled Brauer algebras}

\begin{thm}
\label{WB-cohom-thm}
Let $r,s\geqslant 1$. For any $\delta\in k$ and for $r\neq s$, there exist isomorphisms of graded $k$-modules
\[\Tor_{\star}^{\mathcal{B}_{r,s}(\delta)}(\1,\1) \cong H_{\star}(\Sigma_{r}\times \Sigma_s,\1) \quad \text{and} \quad \Ext_{\mathcal{B}_{r,s}(\delta)}^{\star}(\1,\1) \cong H^{\star}(\Sigma_{r}\times \Sigma_s,\1).\]
Furthermore, when $r=s$, these isomorphisms hold in the range $0\leqslant \star\leqslant \frac{r+s}{2}-1$.  
\end{thm}
\begin{proof}
This follows from Proposition \ref{stable-isos-prop} with the idempotent left covers of Proposition \ref{WB-odd-idem-cover-prop} (when $r\neq s$) and Proposition \ref{WB-even-idem-cover-prop} (when $r=s$), together with the isomorphism of $k$-algebras $\mathcal{B}_{r,s}(\delta)/I_{r+s-1}\cong k[\Sigma_r\times \Sigma_s]$.
\end{proof}

We will now show that the range found in Theorem \ref{WB-cohom-thm} for the case $r=s$ cannot be extended in general by showing that it is sharp for the case of $\mathcal{B}_{1,1}(\delta)$.

\begin{prop}
\label{1+1-prop}
There is an isomorphism of $k$-algebras $\mathcal{B}_{1,1}(\delta)\cong \mathcal{TL}_2(\delta)$. In particular, there exist isomorphisms of $k$-modules
\[\Tor_q^{\mathcal{B}_{1,1}(\delta)}(\1,\1) \cong \begin{cases}
    k & q=0\\
    k/\delta k & q>0,\,\text{$q$ odd}\\
    k_{\delta} & q>0,\, \text{$q$ even}
\end{cases}
\]
and 
\[\Ext_{\mathcal{B}_{1,1}(\delta)}^q(\1,\1) \cong \begin{cases}
    k & q=0\\
    k/\delta k & q>0,\,\text{$q$ even}\\
    k_{\delta} & q>0,\, \text{$q$ odd}
\end{cases}\]
where $k_{\delta}$ is the kernel of the map $k\rightarrow k$ given by multiplication by $\delta$.
\end{prop}
\begin{proof}
We note that both algebras have two basis diagrams. The isomorphism of $k$-algebras is determined by sending the identity diagram to the identity diagram, the non-identity diagram to the non-identity diagram and noting in both cases that the non-identity diagram squares to $\delta$ lots of itself.

The isomorphisms on (co)homology follow from the isomorphism of $k$-algebras together with the calculations of Boyd and Hepworth \cite[Proposition 7.1]{BH1}.
\end{proof}

\begin{cor}
The range in Theorem \ref{WB-cohom-thm} is sharp when $r=s=1$.    
\end{cor}
\begin{proof}
Theorem \ref{WB-cohom-thm} tells us that $\Tor_q^{\mathcal{B}_{1,1}(\delta)}(\1,\1)$ and $\Ext_{\mathcal{B}_{1,1}(\delta)}^q(\1,\1)$ agree with the (co)homology of $\Sigma_1\times \Sigma_1$, namely the trivial group, when $q=0$. However, this isomorphism does not extend beyond $q=0$. The (co)homology of $\Sigma_1\times \Sigma_1$ is trivial in positive degrees, whereas Proposition \ref{1+1-prop} tells us that the (co)homology of $\mathcal{B}_{1,1}(\delta)$ is non-trivial in positive degrees. 
\end{proof}

\begin{rem}
    The question as to whether the range found in Theorem \ref{WB-cohom-thm} is sharp for $r=s\geqslant 2$ remains open. This is similar to the case for Brauer algebras $\mathcal{B}_{n}(\delta)$ with $n$ even. 
\end{rem}


\begin{thebibliography}{CDVPSW08}

\bibitem[BJSH24]{BJSH}
{\scshape Balanz{\'o}-Juand{\'o}, M.; Studzi{\'n}ski, M.; Huber, F.} 
Positive maps from the walled {Brauer} algebra.
{\em J. Phys. A, Math. Theor.} {\bf 57} (1983), no. 11, Id/No 115202. 
\mrev{4719021}, 
\zbl{1547.15003},
\doi{10.1088/1751-8121/ad2b86}.

\bibitem[BCHLLS94]{BCHLLS}
{\scshape Benkart, G.; Chakrabarti, M.; Halverson, T.; Leduc, R.; Lee, C.; Stroomer, J.} 
Tensor product representations of general linear groups and their connections with {B}rauer algebras.
{\em J. Algebra} {\bf 166} (1994), no. 3, 529--567. 
\mrev{1280591}, 
\zbl{0815.20028},
\doi{10.1006/jabr.1994.1166}.

\bibitem[BH14]{BHal}
{\scshape Benkart, G.; Halverson, T.} 
Motzkin algebras.
{\em Eur. J. Comb.} {\bf 36} (2014), no. 2, 473--502. 
\mrev{3131911}, 
\zbl{1284.05333},
\doi{10.1016/j.ejc.2013.09.010}.

\bibitem[BM13]{BM}
{\scshape Benkart, G.; Moon, D.} 
Planar rook algebras and tensor representations of {{\(\mathfrak{gl}(1|1)\)}}.
{\em Commun. Algebra} {\bf 41} (2013), no. 7, 2405--2416. 
\mrev{3169400}, 
\zbl{1269.05115},
\doi{10.1080/00927872.2012.658533}.


\bibitem[Ben91]{Benson1}
{\scshape Benson, D.J.} 
Representations and cohomology. {I} 
Cambridge Studies in Advanced Mathematics, 30. 
{\em Cambridge University Press, Cambridge}, 1991. 
xii+224 pp. ISBN: 0-521-36134-6.
\mrev{1110581}, 
\zbl{0718.20001},
\doi{10.1017/CBO9780511623615}.

\bibitem[BBRWS25]{BBRWS}
{\scshape Boyd, R.; Boyde, G.; Randal-Williams, O.; Sroka, R.} 
Even Temperley--Lieb algebras and the dga of planar loops.
arXiv eprint: 2511.08550. (2025).
\doi{10.48550/arXiv.2511.08550}.


\bibitem[BH24]{BH1}
{\scshape Boyd, R.; Hepworth, R.} 
The homology of the {Temperley}-{Lieb} algebras.
{\em Geom. Topol.} {\bf 28} (2024), no. 3, 1437--1499. 
\mrev{4746417}, 
\zbl{1540.20103},
\doi{10.2140/gt.2024.28.1437}.


\bibitem[BHP21]{BHP}
{\scshape Boyd, R.; Hepworth, R.; Patzt, P.} 
The homology of the {Brauer} algebras.
{\em Sel. Math., New Ser.} {\bf 27} (2021), no. 5, Id/No 85. 
\mrev{4304560}, 
\zbl{1484.20088},
\doi{10.1007/s00029-021-00697-4}.


\bibitem[BHP23]{BHP2}
{\scshape Boyd, R.; Hepworth, R.; Patzt, P.} 
The homology of the partition algebras.
{\em Pac. J. Math.} {\bf 327} (2023), no. 1, 1--27. 
\mrev{4705889}, 
\zbl{1537.16011},
\doi{10.2140/pjm.2023.327.1}.



\bibitem[Boy24]{Boyde2}
{\scshape Boyde, G.} 
Stable homology isomorphisms for the partition and {J}ones annular algebras.
{\em Sel. Math. New Ser.} {\bf 30} (2024), no. 5, Id/No 103. 
\mrev{4818537}, 
\zbl{1556.16007},
\doi{10.1007/s00029-024-00992-w}.



\bibitem[Boy25]{Boyde}
{\scshape Boyde, G.} 
Idempotents and homology of diagram algebras.
{\em Math. Ann.} {\bf 391} (2025), no. 2, 2173--2207. 
\mrev{4853016}, 
\zbl{1561.16006},
\doi{10.1007/s00208-024-02960-3}.


\bibitem[Boy25a]{Boyde-planar}
{\scshape Boyde, G.} 
The dga of planar loops when $2n=4$.
arXiv eprint: 2511.15653. (2025).
\doi{10.48550/arXiv.2511.15653}.



\bibitem[Bra37]{Brauer}
{\scshape Brauer, R.} 
On algebras which are connected with the semisimple continuous groups.
{\em Ann. Math. (2)} {\bf 38} (1937), no. 4, 857--872. 
\mrev{1503378}, 
\zbl{0017.39105},
\doi{10.2307/1968843}.


\bibitem[BS12]{Brundan-Stroppel}
{\scshape Brundan, J.; Stroppel, C.} 
Gradings on walled {Brauer} algebras and {Khovanov}'s arc algebra.
{\em Adv. Math.} {\bf 231} (2012), no. 2, 709--773. 
\mrev{0710056}, 
\zbl{0507.57010},
\doi{10.4310/jdg/1214437665}.


\bibitem[Cam24]{Campbell}
{\scshape Campbell, J.M.} 
Alternating submodules for partition algebras, rook algebras, and rook-{Brauer} algebras.
{\em J. Pure Appl. Algebra} {\bf 228} (1983), no. 1, Id/No 107452. 
\mrev{4603004}, 
\zbl{1532.16011},
\doi{10.1016/j.jpaa.2023.107452}.

\bibitem[CDVPSW24]{CVPSW}
{\scshape Chavli, E.; De Visscher, M.; Parker, A.; Salmon, S.; Wilson, U.} 
The center of the walled brauer algebra {$B_{r,1}(\delta)$}.
arXiv eprint: 2412.14716. (2024).
\doi{10.48550/arXiv.2412.14716}.


\bibitem[CDVDM08]{CDDM}
{\scshape Cox, A.; De~Visscher, M.; Doty, S,; Martin, P.} 
On the blocks of the walled {Brauer} algebra.
{\em J. Algebra} {\bf 320} (2008), no. 1, 169--212. 
\mrev{2417984}, 
\zbl{1196.20004},
\doi{10.1016/j.jalgebra.2008.01.026}.

\bibitem[CG26]{Cranch-Graves}
{\scshape Cranch, J.; Graves, D.} 
Cohomology of coloured partition algebras. 
To appear in {\em Homology Homotopy Appl.} arXiv eprint: 2411.11776. (2026) 
\doi{10.48550/arXiv.2411.11776}.




\bibitem[DFGG97]{DGG}
{\scshape Di~Francesco, P.; Golinelli, O.; Guitter, E.} 
Meanders and the {Temperley}-{Lieb} algebra.
{\em Commun. Math. Phys.} {\bf 186} (1997), no. 1, 1--59. 
\mrev{1462755}, 
\zbl{0873.05008},
\doi{10.1007/BF02885671}.


\bibitem[DG23]{Doty_Giaquinto}
{\scshape Doty, S.; Giaquinto, A.} 
The partial {Temperley}-{Lieb} algebra and its representations.
{\em J. Comb. Algebra} {\bf 7} (2023), no. 3-4, 401--439. 
\mrev{4662317}, 
\zbl{1527.05174},
\doi{10.4171/JCA/74}.

\bibitem[DG24]{DG-twin}
{\scshape Doty, S.; Giaquinto, A.} 
Schur-{Weyl} duality for twin groups.
{\em Transform. Groups} {\bf 29} (2024), no. 2, 621--645. 
\mrev{4771817}, 
\zbl{1552.20178},
\doi{10.1007/s00031-022-09708-w}.


\bibitem[FG25]{Fisher-Graves}
{\scshape Fisher, A.; Graves, D.} 
Cohomology of {Tanabe} algebras.
{\em Extr. Math.} {\bf 40} (2025), no. 2, 235--252. 
\mrev{5010938}, 
\zbl{},
\doi{10.17398/2605-5686.40.2.235}.

\bibitem[FG26]{FG-dilute}
{\scshape Fisher, A.; Graves, D.} 
Cohomology of dilute Temperley--Lieb algebras.
{\em Can. Math. Bull.} (2026). 
\doi{10.4153/S0008439526101829}.


\bibitem[FHH09]{FHH-planar-rook}
{\scshape Flath, D.; Halverson, T.; Herbig, K.} 
The planar rook algebra and {Pascal}'s triangle.
{\em Enseign. Math. (2)} {\bf 55} (2009), no. 1-2, 77--92. 
\mrev{2541502}, 
\zbl{1209.20004},
\doi{10.4171/LEM/55-1-3}.

\bibitem[GL96]{GL-cellular}
{\scshape Graham, J.J.; Lehrer, G.I.} 
Cellular algebras.
{\em Invent. Math.} {\bf 123} (1996), no. 1, 1--34. 
\mrev{1376244}, 
\zbl{0853.20029},
\doi{10.1007/BF01232365}.

\bibitem[Hal96]{Halverson-walled}
{\scshape Halverson, T.} 
Characters of the centralizer algebras of mixed tensor representations of {$GL(r,\mathbb{C})$} and the quantum group {${U}_q(g\ell(r,\mathbb{C}))$}.
{\em Pacific J. Math.} {\bf 174} (1996), no. 2, 359--410. 
\mrev{1405593}, 
\zbl{0871.17010},
\doi{10.2140/pjm.1996.174.359}.


\bibitem[Hd14]{HD}
{\scshape Halverson, T.; delMas, E.} 
Representations of the rook-{Brauer} algebra.
{\em Commun. Algebra} {\bf 42} (2014), no. 1, 423--443. 
\mrev{3169580}, 
\zbl{1291.05215},
\doi{10.1080/00927872.2012.716120}.

\bibitem[HR01]{HalvR}
{\scshape Halverson, T.; Ram, A.} 
{{\(q\)}}-rook monoid algebras, {Hecke} algebras, and {Schur}-{Weyl} duality.
{\em J. Math. Sci., New York} {\bf 121} (2001), no. 3, 2419--2436. 
\mrev{1879072}, 
\zbl{1142.20300},
\doi{10.1023/B:JOTH.0000024623.99412.13}.



\bibitem[Jon83]{Jones1}
{\scshape Jones, V.F.R.} 
Index for subfactors.
{\em Invent. Math.} {\bf 72} (1983), no. 1, 1--25. 
\mrev{0696688}, 
\zbl{0508.46040},
\doi{10.1007/BF01389127}.


\bibitem[Jon85]{Jones2}
{\scshape Jones, V.F.R.} 
A polynomial invariant for knots via von {Neumann} algebras.
{\em Bull. Am. Math. Soc., New Ser.} {\bf 12} (1985), no. 1, 103--111. 
\mrev{0766964}, 
\zbl{0564.57006},
\doi{10.1090/S0273-0979-1985-15304-2}.

\bibitem[JY21]{JoYa}
{\scshape Jones, V.F.R.; Yang, J.} 
Motzkin algebras and the {{\(A_n\)}} tensor categories of bimodules.
{\em Int. J. Math.} {\bf 32} (2021), no. 10, Id/No 2150077. 
\mrev{4311719}, 
\zbl{1486.46068},
\doi{10.1142/S0129167X21500774}.

\bibitem[JK20]{JK}
{\scshape Jung, J.H.; Kim, M.} 
Supersymmetric polynomials and the center of the walled {Brauer} algebra.
{\em Algebr. Represent. Theory} {\bf 23} (2020), no. 5, 1945--1975. 
\mrev{4140060}, 
\zbl{1474.2001},
\doi{10.1007/s10468-019-09922-3}.

\bibitem[Kau90]{Kauffmann1}
{\scshape Kauffman, L.H.} 
An invariant of regular isotopy.
{\em Trans. Am. Math. Soc.} {\bf 318} (1990), no. 2, 417--471. 
\mrev{0958895}, 
\zbl{0763.57004},
\doi{10.2307/2001315}.

\bibitem[Koi89]{Koike}
{\scshape Koike, K.} 
On the decomposition of tensor products of the representations of the classical groups: by means of the universal characters.
{\em Adv. Math.} {\bf 74} (1989), no. 1, 57--86. 
\mrev{0991410}, 
\zbl{0681.20030},
\doi{10.1016/0001-8708(89)90004-2}.

\bibitem[MM14]{MM}
{\scshape Martin, M.; Mazorchuk, V.} 
On the representation theory of partial {Brauer} algebras.
{\em Q. J. Math.} {\bf 65} (2014), no. 1, 225--247. 
\mrev{3179659}, 
\zbl{1354.16016},
\doi{10.1093/qmath/has043}.

\bibitem[Nik06]{Nikitin}
{\scshape Nikitin, P.P.} 
The centralizer algebra of the diagonal action of the group {{\(\text{GL}_n(\mathbb C)\)}} in a mixed tensor space.
{\em Zap. Nauchn. Semin. POMI} {\bf 331} (2006), no. 1, 170--198.  
\zbl{1135.20032}.

\bibitem[RSA14]{RSA}
{\scshape Ridout, D.; Saint-Aubin, Y.} 
Standard modules, induction and the structure of the {Temperley}-{Lieb} algebra.
{\em Adv. Theor. Math. Phys.} {\bf 18} (2014), no. 5, 957--1041. 
\mrev{3281274}, 
\zbl{1308.82015},
\doi{10.4310/ATMP.2014.v18.n5.a1}.

\bibitem[Sro22]{Sroka}
{\scshape Sroka, R.J.} 
The homology of a {T}emperley-{L}ieb algebra on an odd number of strands.
{\em Algebr. Geom. Topol.} {\bf 24} (2024), no. 6, 3527--3542. 
\mrev{4812225}, 
\zbl{1556.16008},
\doi{10.2140/agt.2024.24.3527}.

\bibitem[Ta25]{KhoaTa}
{\scshape Ta, K.} 
Homology of rook-Brauer algebras and Motzkin algebras, 2025.
arXiv eprint: 2505.21977. (2025).
\doi{10.48550/arXiv.2505.21977}.


\bibitem[TL71]{TL}
{\scshape Temperley, H.V.N; Lieb, E.H.} 
Relations between the `percolation' and `colouring' problem and other graph-theoretical problems associated with regular planar lattices: some exact results for the 'percolation' problem.
{\em Proc. R. Soc. Lond., Ser. A} {\bf 322} (1971), no. 1, 251--280. 
\mrev{0498284}, 
\zbl{0211.56703},
\doi{10.1098/rspa.1971.0067}.

\bibitem[Tur90]{Turaev}
{\scshape Turaev, V.G.} 
Operator invariants of tangles, and {R}-matrices.
{\em Math. USSR, Izv.} {\bf 35} (1990), no. 2, 411--444. 
\mrev{1024455}, 
\zbl{0707.57003},
\doi{10.1070/IM1990v035n02ABEH000711}.

\bibitem[Wes95]{Westbury}
{\scshape Westbury, B.W.} 
The representation theory of the {Temperley}-{Lieb} algebras.
{\em Math. Z.} {\bf 219} (1995), no. 4, 539--565. 
\mrev{1343661}, 
\zbl{0840.16008},
\doi{10.1007/BF02572380}.

\bibitem[Xia16]{Xiao}
{\scshape Xiao, Z.K.} 
On tensor spaces for rook monoid algebras.
{\em Acta Math. Sin., Engl. Ser.} {\bf 32} (2016), no. 5, 607--620. 
\mrev{3483932}, 
\zbl{1353.20033},
\doi{10.1007/s10114-016-5546-8}.

\end{thebibliography}
\end{document}